\documentclass[11pt,a4paper]{article}
\pdfoutput=1

\usepackage{amssymb}
\usepackage{amsmath}
\usepackage{amsfonts}
\usepackage{bbm}
\usepackage{amsthm}
\usepackage{mathrsfs}
\usepackage{hyperref}
\usepackage{color}
\usepackage[margin=2.41cm]{geometry}
\usepackage[all,cmtip]{xy}
\usepackage[utf8]{inputenc}
\usepackage{graphicx}
\usepackage{varwidth}
\usepackage{comment}
\usepackage{enumitem}

\usepackage{mathtools}

\usepackage{upgreek}
\usepackage{rotating}

\usepackage{tikz}
\usetikzlibrary{shapes.geometric}

\usepackage{tikz-cd}
\usetikzlibrary{matrix,arrows,calc,decorations.pathmorphing,fit,shapes.geometric,decorations.pathreplacing,positioning}

\definecolor{darkred}{rgb}{0.8,0.1,0.1}
\hypersetup{
	colorlinks=true,         
	urlcolor=darkred,
	linkcolor=darkred,
	citecolor=blue,
}

\theoremstyle{plain}
\newtheorem{theo}{Theorem}[section]
\newtheorem{lem}[theo]{Lemma}
\newtheorem{propo}[theo]{Proposition}
\newtheorem{cor}[theo]{Corollary}

\theoremstyle{definition}
\newtheorem{defi}[theo]{Definition}

\newtheorem{exx}[theo]{Example}
\newenvironment{ex}
{\pushQED{\qed}\exx}
{\popQED\endexx}

\newtheorem{remm}[theo]{Remark}
\newenvironment{rem}
{\pushQED{\qed}\remm}
{\popQED\endremm}

\newtheorem{convv}[theo]{Convention}
\newenvironment{conv}
{\pushQED{\qed}\convv}
{\popQED\endconvv}

\numberwithin{equation}{section}

\def\nn{\nonumber}

\def\bbK{\mathbb{K}}
\def\bbR{\mathbb{R}}

\def\bbZ{\mathbb{Z}}
\def\bbT{\mathbb{T}}

\def\Hom{\mathrm{Hom}}
\def\hom{\underline{\mathrm{hom}}}

\def\Sym{\mathrm{Sym}}

\def\odd{-}
\def\even{+}

\def\id{\mathrm{id}}

\def\dd{\mathrm{d}}

\def\1{I}
\def\oone{\mathbbm{1}}

\def\Vec{\mathbf{Vec}}
\def\Ch{\mathbf{Ch}}

\def\ad{\mathrm{ad}}
\def\g{\mathfrak{g}}
\def\h{\mathfrak{h}}

\def\Pol{\mathrm{Pol}}
\def\CE{\mathrm{CE}}

\newcommand{\pb}[2]{[\![#1,#2]\!]}
\newcommand{\pbbig}[2]{\big[\!\!\big[#1,#2\big]\!\!\big]}
\newcommand{\du}[2]{\langle #1,#2 \rangle}

\DeclareMathOperator*{\Mwedge}{\text{\raisebox{0.25ex}{\scalebox{0.7}{${\bigwedge}_{}$}}}}

\def\sk{\vspace{2mm}}

\makeatletter
\let\@fnsymbol\@alph
\makeatother

\title{%
Shifted Poisson structures on higher Chevalley-Eilenberg algebras
}

\author{%
Cameron Kemp$^{1,a}$, Robert Laugwitz$^{1,b}$\ and\ Alexander Schenkel$^{1,2,c}$\vspace{4mm}\\
{\small ${}^1$ School of Mathematical Sciences, University of Nottingham,}\\
{\small University Park, Nottingham NG7 2RD, United Kingdom.}\vspace{2mm}\\
{\small ${}^2$ Dipartimento di Matematica, Universit{\`a} di Trento,}\\
{\small Via Sommarive 14, 38123 Povo (Trento), Italy.}\vspace{4mm}\\
{\small \begin{tabular}{ll}
Email: & ${}^a$~\href{mailto:cameron.kemp@nottingham.ac.uk}{\texttt{cameron.kemp@nottingham.ac.uk}}\\
&${}^b$~\href{mailto:robert.laugwitz@nottingham.ac.uk}{\texttt{robert.laugwitz@nottingham.ac.uk}}\\
& ${}^c$~\href{mailto:alexander.schenkel@unitn.it}{\texttt{alexander.schenkel@unitn.it}}
\vspace{2mm}
\end{tabular}
}
}

\date{January 2026}

\begin{document}

\maketitle

\begin{abstract}
\noindent This paper develops a graphical calculus to determine the $n$-shifted Poisson structures on finitely generated semi-free commutative differential graded algebras. When applied to the Chevalley-Eilenberg algebra of an ordinary Lie algebra, we recover Safronov's result that the $(n=1)$- and $(n=2)$-shifted Poisson structures in this case are given by quasi-Lie bialgebra structures and, respectively, invariant symmetric tensors. We generalize these results to the Chevalley-Eilenberg algebra of a Lie $2$-algebra and obtain $n\in\{1,2,3,4\}$ shifted Poisson structures in this case, which we interpret as semi-classical data of `higher quantum groups'.
\end{abstract}
\vspace{-1mm}

\paragraph*{Keywords:} derived algebraic geometry, shifted Poisson structures, 
$L_\infty$-algebras, $L_\infty$-bialgebras 
\vspace{-2mm}

\paragraph*{MSC 2020:} 14A30, 17B55, 17B62
\vspace{-2mm}

\tableofcontents


\section{\label{sec:intro}Introduction and summary}
The concept of Poisson structures on manifolds or algebraic varieties is ubiquitous in mathematics and physics.
Classical physical systems, when described within the Hamiltonian formalism,
come equipped with a canonical Poisson bracket $\{\,\cdot\,,\,\cdot\,\}:A\otimes A\to A$
on the algebra $A$ of functions on their phase space. This Poisson structure plays multiple important roles:
From the point of view of dynamics, the derivation $\{\,\cdot\,,H\}:A\to A$ defined by 
inserting the Hamiltonian $H\in A$ of the system into the Poisson bracket generates the 
time evolution of observables $O\in A$ through Hamilton's equations $\tfrac{d}{dt}O(t) = \{O(t),H\}$.
From the point of view of quantization, the Poisson structure provides an `initial datum'
for the deformation quantization problem which consists of deforming the commutative algebra 
$A$ of classical observables into a noncommutative algebra $A_\hbar$ of quantum observables.
We refer the reader to \cite{Poisson} for a thorough introduction to Poisson structures
and to \cite{Bordemann} and \cite{Esposito} for introductions to deformation quantization.
\sk

The theory of Poisson structures becomes even richer in (differential) graded
geometric contexts in which the algebra $A= \bigoplus_{i\in\bbZ}A^i$ 
is graded. This is due to the fact that in this case
the Poisson bracket $\{\,\cdot\,,\,\cdot\,\}:A\otimes A\to A$ can carry
a non-trivial degree, which alters both its geometric and deformation theoretic aspects.
Adopting the shifting conventions from the derived algebraic geometry 
literature \cite{PTVV,CPTVV} and \cite{Pridham,PridhamOutline},
an $n$-shifted Poisson structure has by definition an underlying Poisson bracket 
$\{\,\cdot\,,\,\cdot\,\}$ of the opposite degree $-n$, where $n\in\bbZ$ is any integer.
One of the prime examples is given by the antibracket
from the Batalin-Vilkovisky (BV) formalism \cite{Batalin,Schwarz}, which according
to these conventions is a $(-1)$-shifted Poisson structure.
While the quantization of ordinary (i.e.\ $0$-shifted) Poisson structures
yields noncommutative algebras, the BV quantization of $(-1)$-shifted 
Poisson structures yields cochain complexes
without any multiplication operation but a deformed differential. 
\sk

A systematic and powerful framework to study shifted Poisson structures
and their deformation quantizations is provided by derived algebraic geometry.
The works \cite{CPTVV,Pridham} develop a precise definition 
for the concept of an $n$-shifted Poisson structure
on a large class of spaces, called derived Artin stacks. 
Furthermore, there are general existence theorems for deformation 
quantizations in the case of positive shifts $n\geq 1$ \cite{CPTVV} 
and also in some specific cases of non-positive shifts $n\in\{-2,-1,0\}$ \cite{Pridham0,Pridham-2,Pridham-1}.
Using a simplified language, the general picture
which emerges from these works is as follows: An $n$-shifted
Poisson structure on a commutative dg-algebra $A$ 
is the `initial datum' for the quantization of $A$ into an $\mathbb{E}_{n+1}$-algebra,
or alternatively for the quantization of the symmetric monoidal dg-category 
${}_{A}\mathbf{Mod}$ of $A$-dg-modules
into an $\mathbb{E}_{n}$-monoidal dg-category. Here
$\mathbb{E}_{n}$ denotes the little $n$-disks operad, which encodes
the behavior of moving around $n$-dimensional disks in $\bbR^n$ without colliding them.
Since $\mathbb{E}_{n}$-operads become increasingly more commutative when increasing $n$,
which loosely speaking results from the additional flexibility in higher-dimensional $\bbR^n$ 
to avoid collisions of disks by moving through the new dimensions, this implies that 
quantizations of $n$-shifted Poisson structures become increasingly more
commutative objects for larger $n$.
\sk

The main aim of our present paper is to apply the concept of
$n$-shifted Poisson structures from derived algebraic geometry
to explore semi-classical aspects of a certain class of `higher quantum groups' 
which are based on higher Lie algebras. We believe that this provides 
a systematic and novel perspective towards a theory of `higher quantum groups' which
complements the existing approaches from the literature, see e.g.\ \cite{Majid2,Zhu,Voronov,Joao,Girelli,Lazarev}.
Our approach is inspired by results of Safronov
\cite{Safronov} which show that the semi-classical
data associated with ordinary quantum groups, namely quasi-Lie bialgebra structures
and invariant symmetric tensors, arise naturally as the $(n=1)$- and $(n=2)$-shifted Poisson 
structures on the classifying stack $\mathrm{B}G=[\mathrm{pt}/G]$ 
of an algebraic group $G$,
or its infinitesimal analogue given by the formal classifying stack 
$\mathrm{B}\g=[\mathrm{pt}/\g]$ of the Lie algebra $\g$ of $G$.
Our concrete proposal is to start from a Lie $N$-algebra $\g$,
which is a higher-categorical generalization of an ordinary Lie algebra,
and to interpret the positively $(n\geq 1)$-shifted Poisson structures on its
formal classifying stack $\mathrm{B}\g=[\mathrm{pt}/\g]$ as the semi-classical
data of `higher quantum groups'. The motivation behind this interpretation is that 
each such datum produces, through the deformation quantization results from \cite{CPTVV}, 
an $\mathbb{E}_n$-monoidal deformation of the symmetric monoidal dg-category 
$\mathbf{dgRep}(\g)$ of representations of $\g$ on cochain complexes, 
which we regard as the representation category of a `higher quantum group'.
Through a higher-categorical variant of Tannakian reconstruction,
which according to our knowledge is currently not yet developed, it might then be
possible to reconstruct from this deformed representation category an
underlying `higher quantum group' in the form of some homotopy algebraic 
generalization of a quasi-triangular Hopf algebra.
In the current paper we address the first step of the program outlined
above, i.e.\ the identification of the semi-classical data for `higher quantum groups'.
The second step concerning categorical deformation quantizations is covered
by the existence results in \cite{CPTVV}, however it is worthwhile to 
note that, due to the use of formality theorems, these quantizations
are rather difficult to describe explicitly. Recent works such as
\cite{KempLaugwitzSchenkel} and \cite{Karlsson} provide attempts 
towards more computational approaches to categorical deformation quantizations, 
but further work is required to make them easier accessible by computational methods.
The last step concerning a higher-categorical generalization of Tannakian reconstruction
is an interesting open problem for future work. It has the potential
to unravel and suggest a suitable homotopy algebraic generalization of the concept of a 
quasi-triangular Hopf algebra, which would open the door for a theory of 
`higher quantum groups' going beyond the examples based on higher Lie algebras from our work.
\sk

We will now explain our results in more detail by outlining the content of this paper.
In Section \ref{sec:prelim}, we recollect some basic definitions about cochain complexes
and commutative dg-algebras in order to fix our notation and to make this paper self-contained.
In Section \ref{sec:polyvectors}, we recall from \cite{CPTVV,Pridham}
the definitions of $n$-shifted polyvectors
and $n$-shifted Poisson structures on a commutative dg-algebra $A$. 
It is important to note that such $n$-shifted Poisson structures
do not only contain the datum of a single binary bracket 
$\{\,\cdot\,,\,\cdot\,\}$, or equivalently a bivector $\pi^{(2)}$,
but rather they contain an a priori infinite tower of homotopy-coherence
data acting as witnesses for the Jacobi identity, see Remark \ref{rem:tower}.
We then specialize these definitions to the cases of finitely generated
free and semi-free commutative dg-algebras and develop a graphical calculus
which is practically useful to visualize and analyze the algebraic properties of $n$-shifted Poisson structures.
(See also Remark \ref{rem:graphicaltoformula} for an illustration of the advantages of our graphical calculus
over element-wise algebraic expressions.)
The main results about our graphical calculus are stated in Corollaries \ref{cor:shiftedPoissonfree}
and \ref{cor:shiftedPoissonsemifree}. We observe in Remark \ref{rem:comparison}
that the derived geometric concept of an $n$-shifted Poisson structure
in this finitely generated semi-free context is equivalent to 
the $L_\infty[0,n-1]$-quasi-bialgebra structures from \cite[Definition 2.5]{Zhu} and 
we thus provide a geometric interpretation for such objects. 
\sk

In Section \ref{sec:higherLie},
we specialize our description of $n$-shifted Poisson structures
further to the case where $A = \CE(\g)$ is the Chevalley-Eilenberg algebra of
a Lie $N$-algebra $\g$. We show in Lemma \ref{lem:degreecounting}
that $n$-shifted Poisson structures on $A = \CE(\g)$ are necessarily trivial
for $n>2N$, which provides us with a useful guiding principle for our exploration.
For ordinary Lie algebras (i.e.\ $N=1$), we recover in Subsection \ref{subsec:Lie}
from our graphical calculus the characterization in \cite{Safronov}
of $(n=1)$- and $(n=2)$- shifted Poisson structures in terms of quasi-Lie bialgebras
and, respectively, invariant symmetric tensors, see
Propositions \ref{prop:2shiftedLie} and \ref{prop:1shiftedLie}.
The focus of Subsection \ref{subsec:2Lie} is on the case of Lie $2$-algebras,
which by Lemma \ref{lem:degreecounting} admit a priori non-trivial
positively $n$-shifted Poisson structures for all $n\in \{1,2,3,4\}$.
In the case of $n\in \{2,3,4\}$, we characterize these shifted Poisson 
structures explicitly by using our graphical calculus and find that they consist
of finitely many data, see Propositions \ref{prop:4shifted2Lie}, \ref{prop:3shifted2Lie}
and \ref{prop:2shifted2Lie}. The case of $n=1$ is exceptional and more difficult
because $1$-shifted Poisson structures consist of a priori infinitely many data, see Remark \ref{rem:1shifted2Lie}.
In the final Subsection \ref{subsec:explicit} we provide a collection of
explicit examples for shifted Poisson structures on the Chevalley-Eilenberg algebras
associated with Abelian Lie $2$-algebras, string Lie $2$-algebras and shifted cotangent Lie $2$-algebras.
Our reason for this particular choice of examples is that these Lie $2$-algebras appear
naturally in mathematical physics, e.g.\ in Abelian gerbes \cite{Brylinski}, in string theory \cite{Saemann}, 
and in higher-dimensional Chern-Simons theory \cite{Zucchini,SchenkelVicedo}.


\section{\label{sec:prelim}Preliminaries}
To fix our notation and conventions, we recall some basic material 
about cochain complexes and (commutative) differential graded algebras.
All vector spaces, algebras, etc., in this paper will be over a fixed field $\bbK$ of characteristic $0$.

\paragraph{Cochain complexes:} A \textit{cochain complex} $V = (V,\dd)$ consists of
a family $V = \{V^i\}_{i\in\bbZ}$ of vector spaces, labeled by integers $i\in\bbZ$ (called degree), 
and a family $\dd = \{\dd^i :  V^i\to V^{i+1}\}_{i\in\bbZ}$ of degree-increasing linear maps
(called differential) which square to zero, i.e.\ $\dd\,\dd=0$. A \textit{cochain map}
$f : V\to W$ is a family $f=\{f^i: V^i\to W^i\}_{i\in\bbZ}$ of degree-preserving linear maps
which commutes with the differentials, i.e.\ $f\,\dd_V = \dd_W\,f$. We denote by $\Ch$ the category
of cochain complexes and cochain maps.
\sk

The category $\Ch$ carries the following standard closed symmetric monoidal structure.
The tensor product $V\otimes W\in\Ch$ of two cochain complexes $V,W\in\Ch$
is given by
\begin{subequations}\label{eqn:Chtensor}
\begin{flalign}
(V\otimes W)^i \,:=\, \bigoplus_{j\in\bbZ}\big(V^j\otimes W^{i-j}\big) \quad,
\end{flalign}
for all $i\in\bbZ$, and the differential
\begin{flalign}
\dd(v\otimes w) \,:=\, (\dd v)\otimes w + (-1)^{\vert v\vert}\,v\otimes (\dd w)\quad,
\end{flalign}
\end{subequations}
for all homogeneous $v\in V$ and all $w\in W$, where $\vert\,\cdot\,\vert$ indicates the degree.
The monoidal unit is given by $\bbK\in\Ch$, regarded as a cochain complex concentrated in degree
$0$ with trivial differential. The symmetric braiding is defined by the Koszul sign rule
\begin{flalign}\label{eqn:Chbraiding}
\gamma \,:\, V\otimes W~\longrightarrow~ W\otimes V~~,\quad v\otimes w~\longmapsto~\gamma(v\otimes w)\,:=\,(-1)^{\vert v\vert\,\vert w\vert}\, w\otimes v\quad,
\end{flalign}
for all homogeneous $v\in V$ and $w\in W$. The internal hom $\hom(V,W)\in\Ch$
between two cochain complexes $V,W\in\Ch$ is given by
\begin{subequations}\label{eqn:Chhom}
\begin{flalign}
\hom(V,W)^i \,:=\,\prod_{j\in\bbZ} \Hom(V^j,W^{j+i})\quad,
\end{flalign}
for all $i\in\bbZ$, where $\Hom$ denotes the vector space of linear maps,
and the differential
\begin{flalign}
\partial L\,:=\, \dd_W\,L - (-1)^{\vert L\vert}\,L\,\dd_V\quad,
\end{flalign}
\end{subequations}
for all homogeneous $L\in\hom(V,W)$. Note that cochain maps $f : V\to W$
are precisely the $0$-cocycles in $\hom(V,W)$, i.e.\ elements $f\in \hom(V,W)^0$ of degree $0$ 
satisfying $\partial f =0$. A \textit{cochain homotopy} between two cochain maps
$f,g:V\to W$ is a degree $-1$ element $h\in \hom(V,W)^{-1}$ such that $\partial h = g-f$.
The degree $i <-1$ elements of $\hom(V,W)$ admit an interpretation in terms of higher cochain homotopies.

\paragraph{Shifting conventions:} Associated to any integer $n\in\bbZ$
is the \textit{$n$-shift} endofunctor $[n]: \Ch\to \Ch$. To a cochain complex
$V\in\Ch$ it assigns the $n$-shifted cochain complex $V[n]\in\Ch$ given 
by $V[n]^i := V^{i+n}$ and $\dd^i_{V[n]}:= (-1)^{n}\,\dd^{i+n}_V$, for all $i\in\bbZ$. 
To a cochain map $f:V\to W$ it assigns the $n$-shifted cochain map $f[n]: V[n]\to W[n]$
given by $\{f[n]^i := f^{i+n} : V^{i+n}\to W^{i+n}\}_{i\in\bbZ}$.
Note that $[0]=\id_{\Ch}$ is the identity and $[n]\, [m] = [n+m]$ under composition. 
Recalling the tensor product \eqref{eqn:Chtensor},
one obtains a natural isomorphism $[n]\cong \bbK[n]\otimes (\,\cdot\,):\Ch\to \Ch$. To keep
track of shifts in element-wise expressions, we will denote elements
in $V[n]$ by $s^{-n}\, v\in V[n]$, where $s^{-n}\in \bbK[n]$ is the degree $-n$ element
determined by the unit of $\bbK$.
\sk

The interplay between shifts and the closed symmetric monoidal structure on $\Ch$ is as follows.
For any $n,m\in\bbZ$ and $V,W\in\Ch$, we have a natural cochain isomorphism
\begin{flalign}\label{eqn:shiftiso}
V[n]\otimes W[m]~\stackrel{\cong}{\longrightarrow}~(V\otimes W)[n+m]~~,\quad s^{-n}\,v\otimes s^{-m}w~\longmapsto~
(-1)^{\vert v\vert\, m}\, s^{-n-m}\,(v\otimes w)\quad,
\end{flalign}
for all homogeneous $v\in V$ and all $w\in W$, which moves the shifts to the left. 
Note that this cochain isomorphism
does not preserve the symmetric braiding on $\Ch$, but rather one has that
\begin{flalign}\label{eqn:shiftisobraiding}
\begin{gathered}
\xymatrix@C=6em{
\ar[d]_-{\cong}V[n]\otimes W[m] \ar[r]^-{\gamma}~&~W[m]\otimes V[n]\ar[d]^-{\cong}\\
(V\otimes W)[n+m]\ar[r]_-{(-1)^{n m}\,\gamma[n+m]}~&~(W\otimes V)[n+m]
}
\end{gathered}\qquad.
\end{flalign}
For the internal hom between $V,W\in\Ch$, we have a natural cochain isomorphism
\begin{flalign}\label{eqn:shiftisohom1}
\hom(V,W)~\stackrel{\cong}{\longrightarrow}~\hom\big(V[n],W[n]\big)~~,\quad
L~\longmapsto~L[n]\,:=\,\id_{\bbK[n]}\otimes L\quad.
\end{flalign}
The tensor product in this expression denotes the one of internal homs,
which when evaluated on elements $s^{-n}\,v\in V[n]$ yields the Koszul sign
$L[n](s^{-n}\,v) := (-1)^{\vert L\vert\,n} \,s^{-n}\,L(v)$.
Furthermore, we have a natural cochain isomorphism
\begin{flalign}
\hom(V,W)[n]~\stackrel{\cong}{\longrightarrow}~\hom\big(V,W[n]\big)
\end{flalign}
that comes without Koszul signs and is given by regarding
$s^{-n}\,L\in \hom(V,W)[n]$ as the element $\widetilde{L}\in \hom\big(V,W[n]\big)$
which is defined by the evaluation $\widetilde{L}(v):= s^{-n}L(v)\in W[n]$, for all $v\in V$.

\paragraph{Commutative differential graded algebras:} A \textit{differential graded algebra} (in short, DGA)
is a monoid object in the symmetric monoidal category $\Ch$. More explicitly, a DGA
is a triple $A= (A,\mu,\eta)$ consisting of a cochain complex $A=(A,\dd)\in\Ch$
and two cochain maps $\mu :A\otimes A\to A\,,~a\otimes a^\prime \mapsto a\,a^\prime$ 
(called multiplication) and $\eta : \bbK\to A\,,~1\mapsto \oone$ (called unit) 
which satisfy the usual associativity and unitality
conditions. A \textit{commutative differential graded algebra} (in short, CDGA)
is a DGA whose multiplication is commutative with respect
to the symmetric braiding \eqref{eqn:Chbraiding} on $\Ch$, i.e.\ $\mu = \mu\,\gamma$
or when evaluated on homogeneous elements $a\,a^\prime = (-1)^{\vert a\vert \,\vert a^\prime\vert}\,a^\prime\,a$.
The class of examples of CDGAs which is relevant for our work is given by
free CDGAs, which are also known as \textit{symmetric algebras}, and their semi-free deformations.
Recall that the symmetric algebra associated with a cochain complex $V\in\Ch$ is given by
the cochain complex
\begin{flalign}
\Sym\, V\,:=\,\bigoplus_{n\geq 0} \Sym^n \,V\,:=\, \bigoplus_{n\geq 0} \big(V^{\otimes n}\big)_{S_n}\,\in\,\Ch\quad,
\end{flalign}
where $V^{\otimes n} := V\otimes V\otimes \cdots\otimes V\in\Ch$ denotes the $n$-fold
tensor power and $\big(V^{\otimes n}\big)_{S_n}\in\Ch$ are the coinvariants (i.e.\ quotient) of the permutation
group action defined by the symmetric braiding \eqref{eqn:Chbraiding}.
The multiplication is defined by $\mu \big( 
[v_{1}\otimes\cdots\otimes v_n]\otimes[v^\prime_1\otimes\cdots\otimes v^\prime_m]\big) = 
[v_{1}\otimes \cdots \otimes v_n\otimes v^\prime_1\otimes \cdots\otimes v^\prime_m]$
and the unit is given by $\eta(1) = 1\in\bbK = \Sym^0 \,V \subseteq \Sym\,V$.
To ease notation, we shall denote elements of the symmetric algebra simply by $v_1\,v_2\cdots v_n\in \Sym\, V$.
\sk

The cochain complex of \textit{derivations} $\bbT_A\in\Ch$ of a CDGA $A=(A,\mu,\eta)$ 
is defined as the subcomplex
\begin{subequations}\label{eqn:DerA}
\begin{flalign}
\bbT_A\,\subseteq\, \hom(A,A)\,\in\,\Ch
\end{flalign}
whose homogeneous elements $D\in\bbT_A$ satisfy the Leibniz rule, i.e.\
\begin{flalign}
D(a\,a^\prime)\,=\, (Da)\,a^\prime + (-1)^{\vert D\vert\, \vert a\vert}\,a\,(Da^\prime)\quad,
\end{flalign}
\end{subequations}
for all homogeneous $a,a^\prime\in A$. The cochain complex $\bbT_A$ is canonically an $A$-dg-module
with left action $A\otimes \bbT_A\to \bbT_A\,,~a\otimes D\mapsto a\, D$ defined
by the evaluation $(a\, D)(a^\prime):= a\,D(a^\prime)$, for all $a^\prime \in A$. 
It further carries the structure of a Lie algebra object in $\Ch$, i.e.\ it is a dg-Lie algebra (or DGLA),
with Lie bracket given by the commutator 
\begin{flalign}\label{eqn:Lie}
[\,\cdot\,,\,\cdot\,]\,:\,\bbT_A\otimes \bbT_A~\longrightarrow~\bbT_A~~,\quad
D\otimes D^\prime \,\longmapsto\,[D,D^\prime]\,:=\, D\,D^\prime - (-1)^{\vert D\vert\,\vert D^\prime\vert}\,D^\prime\,D\quad,
\end{flalign}
for all homogeneous $D,D^\prime\in\bbT_A$. For later reference, let us spell out in 
detail the properties which the Lie bracket satisfies:
\begin{subequations}\label{eqn:Lieproperties}
\begin{itemize}
\item[(i)] Antisymmetry: For all homogeneous $D,D^\prime\in \bbT_A$,
\begin{flalign}\label{eqn:Lieantisym}
[D,D^\prime] \,=\,-(-1)^{\vert D\vert\,\vert D^\prime\vert}\,[D^\prime,D]\quad.
\end{flalign}

\item[(ii)] Derivation property: For all homogeneous $a\in A$ and $D,D^\prime \in\bbT_A$,
\begin{flalign}\label{eqn:Liederivation}
[D,a\,D^\prime]\,=\,D(a)\,D^\prime + (-1)^{\vert D\vert\,\vert a\vert}\, a\,[D,D^\prime]\quad.
\end{flalign}

\item[(iii)] Jacobi identity: For all homogeneous $D,D^\prime,D^{\prime\prime}\in\bbT_A$,
\begin{flalign}\label{eqn:LieJacobi}
\big[D,[D^\prime,D^{\prime\prime}]\big]\,=\, \big[[D,D^\prime],D^{\prime\prime}\big] + (-1)^{\vert D\vert\,\vert D^\prime\vert }\,\big[D^\prime,[D,D^{\prime\prime}]\big]\quad.
\end{flalign}
\end{itemize}
\end{subequations}
In the case of a free CDGA $A = \Sym\,V$, each derivation is completely determined by its restriction
to the generators $V\subseteq \Sym\,V$, i.e.\ we have a cochain isomorphism
\begin{flalign}\label{eqn:derfree}
\bbT_{\Sym\,V}\,\stackrel{\cong}{\longrightarrow}\,\hom\big(V,\Sym\,V\big)\quad.
\end{flalign}
The inverse is given explicitly by extending each homogeneous $L\in \hom(V,\Sym\,V)$ to a derivation
$D_L$ on $\Sym\,V$ via the Leibniz rule $D_L(v_1\cdots v_n):= \sum_{i=1}^n 
(-1)^{\vert L\vert\,\sum_{j=1}^{i-1}\vert v_j\vert}\,v_1\cdots v_{i-1}\,L(v_i)\,v_{i+1}\cdots v_n$, 
for all homogeneous $v_1,\dots,v_n\in V$.
\sk

We conclude this section by recalling a standard deformation construction for CDGAs.
A degree $1$ derivation $\alpha\in \bbT_A^1$ on a CDGA $A=(A,\mu,\eta)$
is called a \textit{Maurer-Cartan element} (in the DGLA $\bbT_A$) if it satisfies
the Maurer-Cartan equation
\begin{flalign}\label{eqn:semifreeMC}
\partial \alpha + \tfrac{1}{2}\,[\alpha,\alpha] \,=\,0\quad,
\end{flalign}
where $\partial$ is the differential \eqref{eqn:Chhom} on $\bbT_A\subseteq \hom(A,A)$
and $[\,\cdot\,,\,\cdot\,]$ is the Lie bracket \eqref{eqn:Lie}.
Given any Maurer-Cartan element $\alpha\in \bbT_A^1$, one may deform
the differential $\dd\in\bbT_A^1$ on $A$ to a new differential
\begin{flalign}\label{eqn:deformeddifferential}
\dd_\alpha \,:=\,\dd + \alpha\,\in\,\bbT_A^1 
\end{flalign}
which, as a consequence of the Maurer-Cartan equation, squares to zero $\dd_\alpha\,\dd_{\alpha}=0$.
One checks that endowing the resulting cochain complex $A_\alpha := (A,\dd_\alpha)\in\Ch$
with the given multiplication and unit defines a new CDGA $A_\alpha = (A_\alpha,\mu,\eta)$.
In the case of free CDGAs $A = \Sym\,V$, one calls
the result $(\Sym\,V)_\alpha$ of such deformations along Maurer-Cartan elements \textit{semi-free CDGAs}.


\section{\label{sec:polyvectors}Shifted polyvectors and shifted Poisson structures}
In this section we first recall the definitions of shifted polyvectors and of shifted 
Poisson structures on a CDGA. These concepts have their origin in derived algebraic geometry
where they provide interesting and fruitful generalizations of the usual polyvectors and Poisson 
structures on algebraic varieties. We refer the reader to \cite{PTVV,CPTVV} and \cite{Pridham,PridhamOutline} 
for the relevant context. Our presentation has intentionally a strong focus on computational details,
such as the explicit Koszul signs arising from shifts of cochain complexes.
This will be important later when we compute and interpret 
shifted Poisson structures on the Chevalley-Eilenberg algebra $\CE(\g)$ of a higher Lie algebra $\g$.
In the case of free and semi-free CDGAs, we provide a convenient graphical calculus
which simplifies the analysis of the individual components of a shifted Poisson structure
and their algebraic properties.

\paragraph{Basic definitions:}
Let us fix an arbitrary integer $n\in\bbZ$, called the shift of polyvectors.
\begin{defi}\label{def:shiftedpoly}
The CDGA of \textit{$n$-shifted polyvectors} on a CDGA $A$ is defined as the relative symmetric algebra
\begin{flalign}
\Pol(A,n)\,:=\, \Sym_A^{}\big(\bbT_A[-n-1]\big)\,=\, \bigoplus_{m\geq 0}\Sym^m_A\big(\bbT_A[-n-1]\big)\,=\,
\bigoplus_{m\geq 0}\big(\bbT_A[-n-1]^{\otimes_A m}\big)_{S_m}
\end{flalign}
on the $(-n-1)$-shift of the $A$-dg-module of derivations $\bbT_A$. 
\end{defi}

The non-negative integer $m$ in this direct sum decomposition is called 
the \textit{weight} of polyvectors. Note that weight and cohomological degree
are two different gradings on $\Pol(A,n)$: An $n$-shifted polyvector of weight $m$ and degree $i$
is an element in $\Sym_A^m\big(\bbT_A[-n-1]\big)^i$.
\sk

The CDGA of $n$-shifted polyvectors $\Pol(A,n)$ inherits from the Lie bracket \eqref{eqn:Lie} 
on the complex of derivations $\bbT_A$ a canonical shifted Poisson bracket
(called Schouten–Nijenhuis bracket), which endows $\Pol(A,n)$ 
with the structure of a $\mathbb{P}_{n+2}$-algebra. We describe this bracket
in terms of an ordinary (i.e.\ degree $0$) Lie bracket
\begin{flalign}\label{eqn:Schoutenbracket}
\pb{\,\cdot\,}{\,\cdot\,} \,:\, \Pol(A,n)[n+1]\otimes\Pol(A,n)[n+1]~\longrightarrow~\Pol(A,n)[n+1]
\end{flalign}
on the $(n+1)$-shift $\Pol(A,n)[n+1]\in\Ch$ of the cochain complex of $n$-shifted polyvectors.
This is defined on weight $\leq 1$ homogeneous $n$-shifted polyvectors 
$a,a^\prime\in \Sym_A^{0}\big(\bbT_A[-n-1]\big)= A$
and $s^{n+1}D,s^{n+1}D^\prime\in \Sym_A^{1}\big(\bbT_A[-n-1]\big) = \bbT_A[-n-1]$ by
\begin{subequations}\label{eqn:Schoutenweight<1}
\begin{flalign}
\pb{s^{-n-1}a}{s^{-n-1}a^\prime}\,&:=\,0 \quad, \\
\pb{s^{-n-1}s^{n+1}D}{s^{-n-1}a}\,&:=\,(-1)^{\vert D\vert\, (n+1)}\,s^{-n-1}D(a) \quad, \\
\pb{s^{-n-1}s^{n+1}D}{s^{-n-1}s^{n+1}D^\prime}\,&:=\,s^{-n-1}s^{n+1}\,[D,D^\prime]\quad,
\end{flalign}
\end{subequations}
where $D(a)$ denotes the evaluation of derivations and $[D,D^\prime]$ is the Lie bracket \eqref{eqn:Lie} on $\bbT_A$, 
and it is then extended as a biderivation to higher weights. The Schouten–Nijenhuis bracket \eqref{eqn:Schoutenbracket}
satisfies similar, but shifted, algebraic 
properties as the Lie bracket on $\bbT_A$ (see \eqref{eqn:Lieproperties}), which explicitly read as follows:
\begin{subequations}\label{eqn:Schoutenproperties}
\begin{itemize}
\item[(i)] Antisymmetry: For all homogeneous $P,Q\in \Pol(A,n)$,
\begin{flalign}\label{eqn:Schoutenantisym}
\pb{s^{-n-1}P}{s^{-n-1}Q} \,=\,-(-1)^{(\vert P\vert-n-1)\,(\vert Q\vert-n-1)}\,\pb{s^{-n-1}Q}{s^{-n-1}P}\quad.
\end{flalign}

\item[(ii)] Derivation property: For all homogeneous $P,Q,R\in \Pol(A,n)$,
\begin{flalign}\label{eqn:Schoutenderivation}
\pb{s^{-n-1}P}{s^{-n-1}Q\,R}\,=\, \pb{s^{-n-1}P}{s^{-n-1}Q}\,R + 
(-1)^{\vert P\vert\,\vert Q\vert}~ Q\,\pb{s^{-n-1}P}{s^{-n-1}R}\quad.
\end{flalign}

\item[(iii)] Jacobi identity: For all homogeneous $P,Q,R\in \Pol(A,n)$,
\begin{multline}\label{eqn:SchoutenJacobi}
\pbbig{s^{-n-1}P}{\pb{s^{-n-1}Q}{s^{-n-1}R}}\,=\,
\pbbig{\pb{s^{-n-1}P}{s^{-n-1}Q}}{s^{-n-1}R}\\[4pt]
+(-1)^{(\vert P\vert -n-1)\,(\vert Q\vert -n-1)}~\pbbig{s^{-n-1}Q}{\pb{s^{-n-1}P}{s^{-n-1}R}}\quad.
\end{multline}
\end{itemize}
\end{subequations}
Note that the Schouten–Nijenhuis bracket $\pb{\,\cdot\,}{\,\cdot\,}$ decreases the weight by $1$, i.e.\ 
for $P\in \Pol(A,n)$ of weight $m_P$ and $Q\in \Pol(A,n)$ of weight $m_Q$,
the weight of $\pb{s^{-n-1}P}{s^{-n-1}Q} $ is $m_P + m_Q -1$.
\sk

The definition of $n$-shifted Poisson structures on a CDGA $A$
uses a completion of the $\mathbb{P}_{n+2}$-algebra of $n$-shifted polyvectors $\Pol(A,n)$ 
in order to avoid bounds on the weights, see e.g.\ \cite{Pridham,PridhamOutline}. 
This completion is obtained by replacing the direct sum
of cochain complexes in Definition \ref{def:shiftedpoly} by a product.
\begin{defi}\label{def:shiftedpolycompleted}
The \textit{completed $n$-shifted polyvectors} on a CDGA $A$ are defined as
\begin{flalign}
\widehat{\Pol}(A,n)\,:=\, \prod_{m\geq 0}\Sym^m_A\big(\bbT_A[-n-1]\big)\quad.
\end{flalign}
A completed $n$-shifted polyvector is thus a formal sum 
$P = \sum_{m\geq 0}P^{(m)}\in \widehat{\Pol}(A,n)$ 
of homogeneous weight components $P^{(m)}\in \Sym^m_A\big(\bbT_A[-n-1]\big)$.
The $\mathbb{P}_{n+2}$-algebra structure on $\Pol(A,n)$ extends 
to the completion $\widehat{\Pol}(A,n)$ by setting
\begin{subequations}\label{eqn:hatPolstructure}
\begin{flalign}
P\,Q\,&:=\,\sum_{m\geq 0} \bigg(\sum_{k+l=m}P^{(k)}\,Q^{(l)}\bigg)\quad,\\
\pb{s^{-n-1}P}{s^{-n-1}Q}\,&:=\,
\sum_{m\geq 0} \bigg(\sum_{k+l-1=m}\pb{s^{-n-1}P^{(k)}}{s^{-n-1}Q^{(l)}}\bigg)\quad,
\end{flalign}
\end{subequations}
for all $P,Q\in \widehat{\Pol}(A,n)$. Note that these are well-defined because 
the weight is bounded from below by $0$ and hence the sums in the parentheses are finite.
\end{defi}

We now define the concept of an $n$-shifted Poisson structure following
\cite[Definition 1.5]{Pridham}, see also \cite[Definition 2.5]{PridhamOutline}.
\begin{defi}\label{def:shiftedPoisson}
An \textit{$n$-shifted Poisson structure} on a CDGA $A$ is a completed $n$-shifted polyvector
\begin{subequations}\label{eqn:MCshiftedPoisson}
\begin{flalign}
\pi\,=\sum_{m\geq 2} \pi^{(m)}\,\in\,\widehat{\Pol}(A,n)^{n+2}
\end{flalign}
of degree $n+2$ and weight $\geq 2$ which satisfies the Maurer-Cartan equation
\begin{flalign}
\partial\big(s^{-n-1}\pi\big) + \tfrac{1}{2}\,\pbbig{s^{-n-1}\pi}{s^{-n-1}\pi}\,=\,0
\end{flalign}
\end{subequations}
with respect to the Schouten–Nijenhuis bracket.
\end{defi}
\begin{rem}\label{rem:tower}
Decomposing $\pi\,=\sum_{m\geq 2} \pi^{(m)}$ into its weight components and using \eqref{eqn:hatPolstructure},
one observes that the Maurer-Cartan equation \eqref{eqn:MCshiftedPoisson} is equivalent
to the following tower of conditions
\begin{flalign}
\nn \partial\big(s^{-n-1}\pi^{(2)}\big)\,&=\,0\quad,\\
\nn \partial\big(s^{-n-1}\pi^{(3)}\big)+ \tfrac{1}{2}\,\pbbig{s^{-n-1}\pi^{(2)}}{s^{-n-1}\pi^{(2)}}\,&=\,0\quad,\\
\nn &\vdots\\
\partial\big(s^{-n-1}\pi^{(m)}\big)+ \tfrac{1}{2}\,\sum_{k+l-1=m}\pbbig{s^{-n-1}\pi^{(k)}}{s^{-n-1}\pi^{(l)}}\,&=\,0\quad,
\end{flalign}
for all $m\geq 3$. The first condition states that the shifted bivector $\pi^{(2)}$
is closed with respect to the differential $\partial$ which is induced by the one of the CDGA $A$. Note that this
condition is automatically satisfied in the case of an ordinary (i.e.\ non-dg) commutative algebra,
so one usually does not encounter it when studying Poisson structures on manifolds or algebraic varieties.
The second condition is a shifted generalization and homotopical relaxation of the usual condition that the 
Schouten–Nijenhuis bracket of the bivector $\pi^{(2)}$ with itself vanishes.
It implies in particular that the binary shifted Poisson bracket on $A$ associated with $\pi^{(2)}$ 
satisfies the Jacobi identity only up to a homotopy determined by the shifted trivector $\pi^{(3)}$. 
The higher weight components $\pi^{(m)}$, for $m\geq 4$, provide a coherent tower of higher homotopies
for the Jacobi identity.
\end{rem}

\paragraph{Finitely generated free CDGAs:} The aim of this paragraph is to specialize the concepts of 
shifted polyvectors and shifted Poisson structures to the case where $A$ is a finitely generated free CDGA. 
To match our present notation and conventions with the ones we use when studying 
higher Lie algebras in Section \ref{sec:higherLie} below, we consider a free CDGA
of the form
\begin{flalign}\label{eqn:freeCDGAg}
A\,=\,\Sym\big(\g^{\ast}[-1]\big)\quad,
\end{flalign}
where $\g^\ast:= \hom(\g,\bbK)\in\Ch$ is the dual of a bounded and degree-wise 
finite-dimensional cochain complex $\g\in\Ch$. We denote the duality pairing
by $\du{\,\cdot\,}{\,\cdot\,} : \g^\ast\otimes \g \to \bbK\,,~\theta \otimes x\mapsto\du{\theta}{x}$.
Later in Section \ref{sec:higherLie}, the cochain complex $\g$ will be equipped with the 
structure of an $L_\infty$-algebra.
\sk

Recalling \eqref{eqn:derfree}, we obtain that the dg-module
of derivations on the CDGA \eqref{eqn:freeCDGAg} is given by
\begin{flalign}\label{eqn:derfreeg}
\bbT_{\Sym(\g^{\ast}[-1])}\,\cong\, \hom\Big(\g^\ast[-1],\Sym\big(\g^{\ast}[-1]\big)\Big)\,\cong\,
\Sym\big(\g^{\ast}[-1]\big)\otimes \g[1]\quad.
\end{flalign}
The evaluation of derivations on the CDGA $\Sym\big(\g^{\ast}[-1]\big)$ is determined 
by shifting and permuting the duality pairing according to
\begin{subequations}\label{eqn:shiftdualitypairing}
\begin{flalign}
\xymatrix{
\g[1]\otimes \g^\ast[-1]\ar[r]^-{\cong}~&~(\g\otimes\g^\ast)[1-1]=\g\otimes\g^\ast \ar[r]^-{\gamma}
~&~\g^\ast\otimes\g \ar[r]^-{{\du{\,\cdot\,}{\,\cdot\,}}}~&~\bbK
}\quad,
\end{flalign}
where the first cochain isomorphism is given in \eqref{eqn:shiftiso} and it moves the shifts to the left.
With a slight abuse of notation, we denote the resulting pairing by the same symbol 
$\du{\,\cdot\,}{\,\cdot\,}$ as the duality pairing. For homogeneous
$s^{-1}\,x\in\g[1]$ and $s\,\theta\in\g^\ast[-1]$, this pairing reads explicitly as
\begin{flalign}
\du{s^{-1}x}{s\,\theta}\,:=\, (-1)^{\vert x\vert}\,(-1)^{\vert x\vert\,\vert\theta\vert} 
\, s^{-1}s\,\du{\theta}{x}\,=\,\du{\theta}{x}\quad, 
\end{flalign}
\end{subequations}
where in the last step we used that $\vert x\vert = -\vert \theta\vert$ whenever $\du{\theta}{x}\neq 0$.
\sk

Using \eqref{eqn:derfreeg}, it follows that the $n$-shifted polyvectors from Definition \ref{def:shiftedpoly}
on the free CDGA $\Sym\big(\g^{\ast}[-1]\big)$ are given by
\begin{subequations}\label{eqn:shiftedpolyg}
\begin{flalign}
\Pol\big(\Sym\big(\g^{\ast}[-1]\big),n\big)\,&\cong\,\Sym\big(\g[-n]\big)\otimes \Sym\big(\g^{\ast}[-1]\big)
\,\cong\, \Sym\Big(\g[-n]\oplus \g^{\ast}[-1]\Big)\quad.
\end{flalign}
Observe that, in addition to the cohomological degree of the underlying cochain complexes, 
there are two additional gradings 
\begin{flalign}
\Pol\big(\Sym\big(\g^{\ast}[-1]\big),n\big)\,=\,
\bigoplus_{m,l\geq 0}\Pol^{m,l}\big(\Sym\big(\g^{\ast}[-1]\big),n\big)\,\cong\,
\bigoplus_{m,l\geq 0}\Big(\Sym^m\big(\g[-n]\big)\otimes \Sym^l\big(\g^{\ast}[-1]\big)\Big)
\end{flalign}
\end{subequations}
given by the weight $m$ of shifted polyvectors
and the symmetric power $l$ in the underlying free algebra $\Sym\big(\g^{\ast}[-1]\big)$.
The completion from Definition \ref{def:shiftedpolycompleted} is then given by
\begin{flalign}\label{eqn:shiftedpolygcomplete}
\widehat{\Pol}\big(\Sym\big(\g^{\ast}[-1]\big),n\big)\,=\,
\prod_{m\geq 0}\bigoplus_{l\geq 0}\Pol^{m,l}\big(\Sym\big(\g^{\ast}[-1]\big),n\big)\quad,
\end{flalign}
i.e.\ only the weight gets completed. As a side remark,
let us note that completing also the symmetric power $l$
in \eqref{eqn:shiftedpolygcomplete} defines the CDGA of 
completed $n$-shifted polyvectors on the completed symmetric algebra $\widehat{\Sym}\big(\g^{\ast}[-1]\big)
:=\prod_{l\geq 0} \Sym^l\big(\g^{\ast}[-1]\big)$.
\sk

In the present case of a free CDGA, 
the Schouten–Nijenhuis bracket \eqref{eqn:Schoutenbracket} admits a simplified description in terms of the 
commutator of a composition operation
\begin{flalign}\label{eqn:composition}
\bullet\,:\,\Pol\big(\Sym\big(\g^{\ast}[-1]\big),n\big)[n+1]\otimes 
\Pol\big(\Sym\big(\g^{\ast}[-1]\big),n\big)[n+1]~\longrightarrow~
\Pol\big(\Sym\big(\g^{\ast}[-1]\big),n\big)[n+1]
\end{flalign}
which we shall now describe. We define the cochain map $\bullet$ on the
homogeneous generators $s\,\theta,s\,\theta^\prime\in \g^\ast[-1]$ 
and $s^nx,s^nx^\prime\in\g[-n]$ of \eqref{eqn:shiftedpolyg} by
\begin{subequations}\label{eqn:compositiongenerators}
\begin{flalign}
(s^{-n-1}s\,\theta)\bullet(s^{-n-1}s\,\theta^\prime)\,&=\,0\quad,\\
(s^{-n-1}s^n x)\bullet (s^{-n-1}s\,\theta)\,&=\,0 \quad,\\
(s^{-n-1}s\,\theta)\bullet (s^{-n-1}s^n x)\,&=\,(-1)^{\vert \theta\vert}~s^{-n-1}\,\du{\theta}{x} \quad,\\
(s^{-n-1}s^n x)\bullet(s^{-n-1}s^n x^\prime)\,&=\,0\quad,
\end{flalign}
\end{subequations}
and extend to $\Pol\big(\Sym\big(\g^{\ast}[-1]\big),n\big)$ as a biderivation, i.e.\
by demanding the properties
\begin{subequations}\label{eqn:compositionder}
\begin{flalign}
\nn (s^{-n-1}P)\bullet(s^{-n-1}Q\,R) \,&=\, \big((s^{-n-1}P)\bullet(s^{-n-1}Q)\big)\,R\\
&\qquad\quad +(-1)^{\vert P\vert\,\vert Q\vert} \,Q \,\big((s^{-n-1}P)\bullet(s^{-n-1}R)\big)\label{eqn:bulletrightder}
\end{flalign}
and
\begin{flalign}
\nn (s^{-n-1}P\,Q)\bullet(s^{-n-1}R) \,&=\, (-1)^{\vert P\vert\,(n+1)}\,P\, \big((s^{-n-1}Q)\bullet(s^{-n-1}R)\big)\\
&\qquad\quad + (-1)^{\vert Q\vert\,(\vert R\vert -n-1)}\,\big((s^{-n-1}P)\bullet(s^{-n-1}R)\big)\,Q\quad,\label{eqn:bulletleftder}
\end{flalign}
\end{subequations}
for all homogeneous $P,Q,R\in \Pol\big(\Sym\big(\g^{\ast}[-1]\big),n\big)$.
\begin{propo}\label{prop:bulletcommutator}
The Schouten–Nijenhuis bracket \eqref{eqn:Schoutenbracket} on the $n$-shifted
polyvectors of a free CDGA $\Sym\big(\g^{\ast}[-1]\big)$ agrees with the commutator
of the composition operation \eqref{eqn:composition}, i.e.\
\begin{flalign}\label{eqn:Schouten=commutator}
\pb{s^{-n-1}P}{s^{-n-1}Q}\,=\, (s^{-n-1}P)\bullet (s^{-n-1}Q) - (-1)^{(\vert P\vert-n-1)\,(\vert Q\vert -n-1)}~(s^{-n-1}Q)\bullet (s^{-n-1}P)\quad,
\end{flalign}
for all homogeneous $P,Q\in \Pol\big(\Sym\big(\g^{\ast}[-1]\big),n\big)$.
\end{propo}
\begin{proof}
The commutator of $\bullet$ is manifestly antisymmetric, as is the Schouten–Nijenhuis bracket 
\eqref{eqn:Schoutenantisym}, and it also satisfies the derivation property \eqref{eqn:Schoutenderivation}.
The latter statement follows from \eqref{eqn:bulletrightder} and the following identity
\begin{flalign}
\nn (s^{-n-1}Q\,R)\bullet(s^{-n-1}P) \,&=\, (-1)^{\vert R\vert\,\vert Q\vert}~(s^{-n-1}R\,Q)\bullet(s^{-n-1}P)\\[4pt]
\nn \,&=\,(-1)^{\vert R\vert\,(\vert Q\vert -n-1)}\,R\,\big((s^{-n-1}Q)\bullet(s^{-n-1}P)\big)\\
\nn &\qquad\quad+ (-1)^{\vert Q\vert\,(\vert R\vert+\vert P\vert-n-1)}\,\big((s^{-n-1}R)\bullet(s^{-n-1}P)\big)\,Q\\[4pt]
\nn \,&=\,(-1)^{\vert R\vert\,(\vert P\vert -n-1)}\,\big((s^{-n-1}Q)\bullet(s^{-n-1}P)\big)\,R\\
 &\qquad\quad + (-1)^{\vert Q\vert\,(n+1)}\,Q\,\big((s^{-n-1}R)\bullet(s^{-n-1}P)\big)\quad,
\end{flalign}
where in the first and third step we used commutativity of the CDGA
$\Pol\big(\Sym\big(\g^{\ast}[-1]\big),n\big)$ and in the second step we used \eqref{eqn:bulletleftder}.
As a consequence of antisymmetry and the derivation property, it suffices to verify the 
identity \eqref{eqn:Schouten=commutator} on the generators
of \eqref{eqn:shiftedpolyg}. For the left-hand side given by the 
Schouten–Nijenhuis bracket, one computes using \eqref{eqn:Schoutenweight<1} and \eqref{eqn:shiftdualitypairing} that
\begin{subequations}\label{eqn:Schouteng}
\begin{flalign}
\pb{s^{-n-1}s\,\theta}{s^{-n-1}s\,\theta^\prime}\,&=\,0\,=\,\pb{s^{-n-1}s^n x}{s^{-n-1}s^n x^\prime}\quad,\\[4pt]
\nn \pb{s^{-n-1}s^n x}{s^{-n-1}s\,\theta}\,&=\,\pb{s^{-n-1}s^{n+1}s^{-1} x}{s^{-n-1}s\,\theta}\,=\,
(-1)^{(\vert x\vert-1)\,(n+1)}\,s^{-n-1}\,\du{s^{-1}x}{s\theta} \\
\, &= \, (-1)^{(\vert x\vert-1)\,(n+1)}\,s^{-n-1}\,\du{\theta}{x} \quad,
\end{flalign}
\end{subequations}
for all homogeneous $s\,\theta,s\,\theta^\prime\in \g^\ast[-1]$ 
and $s^nx,s^nx^\prime\in\g[-n]$. Comparing with \eqref{eqn:compositiongenerators}
one immediately observes that this coincides with the commutator of $\bullet$,
which implies that \eqref{eqn:Schouten=commutator} holds true for 
generators and hence for all shifted polyvectors.
\end{proof}

\begin{cor}\label{cor:Poissonviabullet}
A completed $n$-shifted polyvector 
$\pi = \sum_{m\geq 2}\pi^{(m)}\in\widehat{\Pol}\big(\Sym\big(\g^{\ast}[-1]\big)\big)^{n+2}$
of degree $n+2$ and weight $\geq 2$ 
is an $n$-shifted Poisson structure in the sense of Definition \ref{def:shiftedPoisson} if and only if
\begin{flalign}\label{eqn:MCbullet}
\partial\big(s^{-n-1}\pi\big)+ (s^{-n-1}\pi)\bullet(s^{-n-1}\pi)\,=\,0\quad.
\end{flalign}
\end{cor}
\begin{proof}
From Proposition \ref{prop:bulletcommutator}, it follows that
\begin{flalign}
\nn\pb{s^{-n-1}\pi}{s^{-n-1}\pi}\,&=\, (s^{-n-1}\pi)\bullet (s^{-n-1}\pi) - (-1)^{(n+2-n-1)^2}~  (s^{-n-1}\pi)\bullet (s^{-n-1}\pi)\\
\,&=\, 2~(s^{-n-1}\pi)\bullet (s^{-n-1}\pi)\quad.
\end{flalign}
This implies that \eqref{eqn:MCbullet} is equivalent to the Maurer-Cartan equation \eqref{eqn:MCshiftedPoisson}.
\end{proof}

\paragraph{Shifted polyvectors as maps:} In Section \ref{sec:higherLie} below,
it will be convenient to identify $n$-shifted polyvectors and 
Poisson structures on the free CDGA $\Sym\big(\g^{\ast}[-1]\big)$ in terms of 
maps $\hom(\g^{\otimes l},\g^{\otimes m})$
between tensor powers of the \textit{unshifted} cochain complex $\g$. This identification involves
Koszul signs, e.g.\ those arising from the shifting isomorphisms \eqref{eqn:shiftiso},
which contribute to the explicit form of the transferred composition operation \eqref{eqn:composition}
and hence the transferred Schouten–Nijenhuis bracket.
The aim of this paragraph is to work out these identifications in detail.
\sk

First, let us observe that, as a consequence of \eqref{eqn:shiftiso} and \eqref{eqn:shiftisobraiding}, 
we have cochain isomorphisms
\begin{subequations}
\begin{flalign}
\Sym\big(\g^\ast[-1]\big)~&\stackrel{\cong}{\longrightarrow}~\bigoplus_{l\geq 0}\big(\text{${\Mwedge}^l$}\g^\ast\big)[-l]\quad,\\
\nn s\theta_1\,\cdots\,s\theta_l~&\longmapsto~(-1)^{\sum_{j=1}^l \vert\theta_j\vert \,(l-j)}~s^{l}\theta_1\cdots\theta_l
\end{flalign}
and
\begin{flalign}
\Sym\big(\g[-n]\big)~&\stackrel{\cong}{\longrightarrow}~\bigoplus_{m\geq 0}\big(\Sym^m_\pm\g\big)[-mn]\quad,\\
\nn s^n x_1\,\cdots\,s^n x_m~&\longmapsto~(-1)^{\sum_{k=1}^m \vert x_k\vert \,(m-k)n}~s^{mn}x_1\cdots x_m\quad,
\end{flalign}
where
\begin{flalign}
\Sym_\pm^{m}\, V\,:=\,\begin{cases}
\Sym^m \,V&~,~~\text{for $n$ even}\quad,\\
{\Mwedge}^m V&~,~~\text{for $n$ odd}\quad,\\
\end{cases}
\end{flalign}
\end{subequations}
denotes the symmetric/exterior powers of a cochain complex $V\in\Ch$.
Using \eqref{eqn:shiftiso} once more, we obtain for the $n$-shifted
polyvectors \eqref{eqn:shiftedpolyg} the cochain isomorphism
\begin{flalign}\label{eqn:polyvectorleftshift}
\Pol\big(\Sym\big(\g^{\ast}[-1]\big),n\big)~\stackrel{\cong}{\longrightarrow}&~
\bigoplus_{m,l\geq 0}\Big(\Sym_\pm^m\,\g \otimes \text{${\Mwedge}^l$}\g^\ast\Big)[-mn-l]\quad,\\
\nn s^n x_1\,\cdots\,s^n x_m\,s\theta_1\,\cdots\,s\theta_l ~\longmapsto&~
(-1)^{\sum_{k=1}^m \vert x_k\vert \,((m-k)n+l)}
~(-1)^{\sum_{j=1}^l \vert\theta_j\vert \,(l-j)}~s^{mn+l}x_1\cdots x_m\,\theta_1\cdots\theta_l\quad.
\end{flalign}
Since this isomorphism is defined weight-wise, it extends to the completed
$n$-shifted polyvectors \eqref{eqn:shiftedpolygcomplete} replacing $\bigoplus_{m\geq 0}$
by $\prod_{m\geq 0}$ on both sides.
\sk

Next, we define a cochain map
\begin{subequations}\label{eqn:polyvector2map}
\begin{flalign}
\Sym_\pm^m\,\g \otimes \text{${\Mwedge}^l$}\g^\ast~\longrightarrow~
\hom\big(\g^{\otimes l},\g^{\otimes m}\big)\quad,
\end{flalign}
for all $m,l\geq 0$, by making use of the duality pairing $\du{\,\cdot\,}{\,\cdot\,}:
\g^\ast\otimes\g\to \bbK$. To an element $B = x_1\cdots x_m\,\theta_1\cdots\theta_l 
\in \Sym_\pm^m\,\g \otimes \text{${\Mwedge}^l$}\g^\ast$, with all $x_k$ and $\theta_j$ homogeneous,
this cochain map assigns the element $L_B\in \hom\big(\g^{\otimes l},\g^{\otimes m}\big)$
which is defined by the evaluation formula
\begin{flalign}
L_B(y_1,\dots,y_l)\,&:=\,\sum_{\sigma\in S_l}\sum_{\rho\in S_m}
(-1)^{\vert x_1\cdots x_m\vert_{\pm}^\rho}~(-1)^{\vert \theta_1\cdots\theta_l\vert^\sigma_{-}}~
(-1)^{\vert y_1\cdots y_l\vert^\mathrm{rev}_+}~\bigotimes_{k=1}^mx_{\rho(k)}~\prod_{j=1}^l\du{\theta_{\sigma(j)}}{y_j}\quad,
\end{flalign}
for all homogeneous $y_1,\dots,y_l\in \g$,
where the Koszul signs are determined from permutations $\rho\in S_m$ in $\Sym_\pm^m\,V$, i.e.\
$v_1\cdots v_m = (-1)^{\vert v_1\cdots v_m\vert^{\rho}_{\pm}}\,v_{\rho(1)}\cdots v_{\rho(m)}$
for all homogeneous $v_1,\dots v_m\in V$ in a cochain complex $V\in\Ch$.
The superscript ${}^{\mathrm{rev}}$ refers to the reversal permutation $(1,2,\dots,l)\to (l,\dots,2,1)$.
The sum over all permutations $\sigma\in S_l,\rho\in S_m$, together with the associated Koszul signs,
encodes the (anti-)symmetry properties of $\Sym_\pm^m\,\g \otimes \text{${\Mwedge}^l$}\g^\ast$
at the level of maps $\hom\big(\g^{\otimes l},\g^{\otimes m}\big)$. Indeed, the cochain map
\eqref{eqn:polyvector2map} defines an isomorphism when corestricting 
\begin{flalign}
\Sym_\pm^m\,\g \otimes \text{${\Mwedge}^l$}\g^\ast~\stackrel{\cong}{\longrightarrow}~
\hom^-_{\pm}\big(\g^{\otimes l},\g^{\otimes m}\big)\,\subseteq\,\hom\big(\g^{\otimes l},\g^{\otimes m}\big)
\end{flalign}
\end{subequations}
to the subcomplex of all $L\in \hom\big(\g^{\otimes l},\g^{\otimes m}\big)$ whose input is totally antisymmetric
and whose output is totally symmetric for $n$ even or antisymmetric for $n$ odd.
The additional Koszul sign $(-1)^{\vert y_1\cdots y_l\vert^\mathrm{rev}_+}$ is a convenient convention
which simplifies some of the signs appearing below.
Combining the above identifications, we obtain a cochain isomorphism
\begin{flalign}\label{eqn:polyvectorhomiso}
\Pol\big(\Sym\big(\g^{\ast}[-1]\big),n\big)\,\cong\,\bigoplus_{m,l\geq 0}\hom^-_{\pm}\big(\g^{\otimes l},\g^{\otimes m}\big)[-mn-l]\quad,
\end{flalign}
which extends to the completed $n$-shifted polyvectors 
\eqref{eqn:shiftedpolygcomplete} replacing $\bigoplus_{m\geq 0}$
by $\prod_{m\geq 0}$ on both sides.

\paragraph{Graphical calculus:} It will be convenient 
to represent elements $L\in \hom^-_{\pm}\big(\g^{\otimes l},\g^{\otimes m}\big)$
graphically by diagrams of the form
\begin{flalign}
\parbox{1.5cm}{\begin{tikzpicture}
\draw[thick] (-0.25,0.75) -- (-0.25,0);
\draw[thick] (0,0.75) -- (0,0);
\draw[thick] (0.25,0.75) -- (0.25,0);
\draw[thick] (-0.375,0) -- (-0.375,-0.75);
\draw[thick] (-0.125,0) -- (-0.125,-0.75);
\draw[thick] (0.125,0) -- (0.125,-0.75);
\draw[thick] (0.375,0) -- (0.375,-0.75);
\draw[thick,fill=white] (-0.6,-0.25) rectangle (0.6,0.25) node[pos=.5] {$L$};
\draw[thick, decorate,decoration={brace,amplitude=5pt}]
  (-0.35,0.85) -- (0.35,0.85) node[midway,yshift=1em]{\text{\footnotesize{$l$ inputs}}};
\draw[thick, decorate,decoration={brace,amplitude=5pt,mirror}]
  (-0.475,-0.85) -- (0.475,-0.85) node[midway,yshift=-1em]{\text{\footnotesize{$m$ outputs}}};
\end{tikzpicture}}~\in~ \hom^-_{\pm}\big(\g^{\otimes l},\g^{\otimes m}\big)
\end{flalign}
which should be read from top to bottom. Total (anti-)symmetry of the inputs and outputs 
under the symmetric braiding $\gamma$ on $\Ch$ is graphically visualized by e.g.\
\begin{subequations}\label{eqn:input/outputpermutations}
\begin{flalign}
(-1)^n~
\parbox{1.25cm}{\begin{tikzpicture}
\draw[thick] (-0.25,0.75) -- (-0.25,0);
\draw[thick] (0,0.75) -- (0,0);
\draw[thick] (0.25,0.75) -- (0.25,0);
\draw[thick] (-0.375,0) -- (-0.375,-0.75);
\draw[thick] (-0.125,0) -- (-0.125,-0.5) -- (0.125,-0.75);
\draw[thick] (0.125,0) -- (0.125,-0.5) -- (-0.125,-0.75);
\draw[thick] (0.375,0) -- (0.375,-0.75);
\draw[thick,fill=white] (-0.6,-0.25) rectangle (0.6,0.25) node[pos=.5] {$L$};
\end{tikzpicture}}
~=~
\parbox{1.25cm}{\begin{tikzpicture}
\draw[thick] (-0.25,0.75) -- (-0.25,0);
\draw[thick] (0,0.75) -- (0,0);
\draw[thick] (0.25,0.75) -- (0.25,0);
\draw[thick] (-0.375,0) -- (-0.375,-0.75);
\draw[thick] (-0.125,0) -- (-0.125,-0.75);
\draw[thick] (0.125,0) -- (0.125,-0.75);
\draw[thick] (0.375,0) -- (0.375,-0.75);
\draw[thick,fill=white] (-0.6,-0.25) rectangle (0.6,0.25) node[pos=.5] {$L$};
\end{tikzpicture}}
~=~(-1)~
\parbox{1.25cm}{\begin{tikzpicture}
\draw[thick] (-0.25,0.75) -- (-0.25,0);
\draw[thick] (0,0.75) -- (0.25,0.5) -- (0.25,0);
\draw[thick] (0.25,0.75) -- (0,0.5) -- (0,0);
\draw[thick] (-0.375,0) -- (-0.375,-0.75);
\draw[thick] (-0.125,0) -- (-0.125,-0.75);
\draw[thick] (0.125,0) -- (0.125,-0.75);
\draw[thick] (0.375,0) -- (0.375,-0.75);
\draw[thick,fill=white] (-0.6,-0.25) rectangle (0.6,0.25) node[pos=.5] {$L$};
\end{tikzpicture}}\quad.
\end{flalign}
More generally, given any permutations $\sigma\in S_l$ and $\rho\in S_m$, we have that
\begin{flalign}
(-1)^{\vert \rho\vert \,n}~~
\parbox{1.25cm}{\begin{tikzpicture}
\draw[thick] (-0.25,0.75) -- (-0.25,0);
\draw[thick] (0,0.75) -- (0,0);
\draw[thick] (0.25,0.75) -- (0.25,0);
\draw[thick] (-0.375,0) -- (-0.375,-1.75);
\draw[thick] (-0.125,0) -- (-0.125,-1.75);
\draw[thick] (0.125,0) -- (0.125,-1.75);
\draw[thick] (0.375,0) -- (0.375,-1.75);
\draw[thick,fill=white] (-0.6,-0.25) rectangle (0.6,0.25) node[pos=.5] {$L$};
\draw[thick,fill=white,dotted] (-0.6,-1.25) rectangle (0.6,-0.75) node[pos=.5] {$\gamma_\rho$};
\end{tikzpicture}}
~=~
\parbox{1.25cm}{\begin{tikzpicture}
\draw[thick] (-0.25,0.75) -- (-0.25,0);
\draw[thick] (0,0.75) -- (0,0);
\draw[thick] (0.25,0.75) -- (0.25,0);
\draw[thick] (-0.375,0) -- (-0.375,-0.75);
\draw[thick] (-0.125,0) -- (-0.125,-0.75);
\draw[thick] (0.125,0) -- (0.125,-0.75);
\draw[thick] (0.375,0) -- (0.375,-0.75);
\draw[thick,fill=white] (-0.6,-0.25) rectangle (0.6,0.25) node[pos=.5] {$L$};
\end{tikzpicture}}
~=~(-1)^{\vert\sigma\vert}~~
\parbox{1.25cm}{\begin{tikzpicture}
\draw[thick] (-0.25,1.75) -- (-0.25,0);
\draw[thick] (0,1.75) -- (0,0);
\draw[thick] (0.25,1.75) -- (0.25,0);
\draw[thick] (-0.375,0) -- (-0.375,-0.75);
\draw[thick] (-0.125,0) -- (-0.125,-0.75);
\draw[thick] (0.125,0) -- (0.125,-0.75);
\draw[thick] (0.375,0) -- (0.375,-0.75);
\draw[thick,fill=white] (-0.6,-0.25) rectangle (0.6,0.25) node[pos=.5] {$L$};
\draw[thick,fill=white,dotted] (-0.6,0.75) rectangle (0.6,1.25) node[pos=.5] {$\gamma_\sigma$};
\end{tikzpicture}}\quad,
\end{flalign}
\end{subequations}
where the dotted boxes represent the action of permutations
via the symmetric braiding $\gamma$ and $\vert\sigma\vert,\vert \rho\vert\in\bbZ_2$ denotes
the signatures of the permutations.
\begin{propo}\label{propo:compositiongraphicalcalculus}
The transfer $\tilde{\bullet}$ of the composition operation $\bullet$ from \eqref{eqn:composition}
along the cochain isomorphism \eqref{eqn:polyvectorhomiso}
reads as follows: For all homogeneous $L\in \hom^-_{\pm}\big(\g^{\otimes l},\g^{\otimes m}\big)$
and $L^\prime\in \hom^-_{\pm}\big(\g^{\otimes l^\prime},\g^{\otimes m^\prime}\big)$,
\begin{multline}
(s^{(m-1)n+l-1}L)\tilde{\bullet} (s^{(m^\prime-1)n+l^\prime-1}L^\prime) \,=\,
(-1)^{\vert L\vert \,((m^\prime-1)n+l^\prime-1)}\,(-1)^{n(m^\prime-1)(l-1)}\,\\[-4pt]
\times\, s^{(m+m^\prime-2)n+l+l^\prime-2}\sum_{\substack{\sigma\in\mathrm{Sh}(l-1,l^\prime) \\ \rho\in\mathrm{Sh}(m,m^\prime-1)}} (-1)^{\vert \sigma\vert}\,(- 1)^{\vert \rho\vert\,n}
~\parbox{1.5cm}{\begin{tikzpicture}[scale=0.75]
\draw[thick] (-0.25,1.75) -- (-0.25,0);
\draw[thick] (0,1.75) -- (0,0);
\draw[thick] (0.25,1.75) -- (0.25,0);
\draw[thick] (-1,1.75) -- (-1,-1.25);
\draw[thick] (-1.25,1.75) -- (-1.25,-1.25);
\draw[thick] (-0.375,-0.25) -- (-0.75,-1);
\draw[thick] (-0.125,0) -- (-0.125,-3);
\draw[thick] (0.125,0) -- (0.125,-3);
\draw[thick] (0.375,0) -- (0.375,-3);
\draw[thick] (-1,-1.5) -- (-1,-3);
\draw[thick] (-1.25,-1.5) -- (-1.25,-3);
\draw[thick] (-0.75,-1.5) -- (-0.75,-3);
\draw[thick,fill=white] (-0.5,-0.25) rectangle (0.5,0.25) node[pos=.5] {\text{\footnotesize{$L^\prime$}}};
\draw[thick,fill=white] (-1.5,-1.5) rectangle (-0.5,-1) node[pos=.5] {\text{\footnotesize{$L$}}};
\draw[thick,fill=white,dotted] (-1.5,0.75) rectangle (0.5,1.25) node[pos=.5] {\text{\footnotesize{$\gamma_\sigma$}}};
\draw[thick,fill=white,dotted] (-1.5,-2.5) rectangle (0.5,-2) node[pos=.5] {\text{\footnotesize{$\gamma_\rho$}}};
\draw[thick, decorate,decoration={brace,amplitude=5pt}]
  (-1.35,1.85) -- (-0.9,1.85) node[midway,yshift=1em]{\text{\footnotesize{${l{-}1}$}}};
\draw[thick, decorate,decoration={brace,amplitude=5pt}]
  (-0.35,1.85) -- (0.35,1.85) node[midway,yshift=1em]{\text{\footnotesize{$l^\prime$}}};
\draw[thick, decorate,decoration={brace,amplitude=5pt,mirror}]
  (-1.35,-3.1) -- (-0.65,-3.1) node[midway,yshift=-1em]{\text{\footnotesize{$m$}}};
\draw[thick, decorate,decoration={brace,amplitude=5pt,mirror}]
  (-0.225,-3.1) -- (0.475,-3.1) node[midway,yshift=-1em]{\text{\footnotesize{${m^\prime{-}1}$}}};
\end{tikzpicture}}\qquad,\label{eqn:compositiononmaps}
\end{multline}
where the sums are over all shuffle permutations.
\end{propo}
\begin{rem}
When evaluated on homogeneous elements $y_1,\dots,y_{l+l^\prime-1}\in\g$, the composed map 
\eqref{eqn:compositiononmaps} reads as follows: Using a Sweedler-like notation
\begin{flalign}
L^\prime(y_1,\dots,y_{l^\prime})\,=\,L^\prime_{(0)}(y_1,\dots,y_{l^\prime})\otimes L^\prime_{(1)}(y_1,\dots,y_{l^\prime})\,\in\,\g\otimes\g^{m^\prime-1}\quad,
\end{flalign}
with summations understood, in order to split off the first tensor factor $\g$, we have that
\begin{multline}\label{eqn:compositiononmapsexplicit}
\big((s^{(m-1)n+l-1}L)\tilde{\bullet} (s^{(m^\prime-1)n+l^\prime-1}L^\prime)\big)(y_1,\dots,y_{l+l^\prime-1})\,=\,
(-1)^{\vert L\vert \,((m^\prime-1)n+l^\prime-1)}\,(-1)^{n(m^\prime-1)(l-1)}\\[5pt]
\times ~s^{(m+m^\prime-2)n+l+l^\prime-2}\sum_{\substack{\sigma\in\mathrm{Sh}(l-1,l^\prime)\\ \rho\in\mathrm{Sh}(m,m^\prime-1)}}(-1)^{\vert y_1\cdots y_{l+l^\prime-1}\vert^\sigma_-}~(- 1)^{\vert \rho\vert\,n}~
(-1)^{\sum_{i=1}^{l-1}\vert y_{\sigma(i)}\vert\,\vert L^\prime\vert}\\
\times~\gamma_\rho\Big( L\Big(y_{\sigma(1)},\dots,y_{\sigma(l-1)}, L^\prime_{(0)}(y_{\sigma(l)},\dots,y_{\sigma(l+l^\prime-1)})\Big)\otimes  L^\prime_{(1)}(y_{\sigma(l)},\dots,y_{\sigma(l+l^\prime-1)}) \Big)\quad.
\end{multline}
The last Koszul sign in the second line 
arises from permuting $L^\prime$ with $y_{\sigma(1)}\otimes \cdots\otimes y_{\sigma(l-1)}$.
\end{rem}
\begin{proof}
The proof is a straightforward but lengthy calculation using the explicit descriptions of the cochain isomorphisms
\eqref{eqn:polyvectorleftshift} and \eqref{eqn:polyvector2map}.
We will not spell out this calculation in full detail, but only give some hints. 
Using these isomorphisms, homogeneous elements $s^{mn+l}L\in \hom^-_{\pm}(\g^{\otimes l},\g^{\otimes m})[-mn-l]$ 
and $s^{m^\prime n+l^\prime}L^\prime\in \hom^-_{\pm}(\g^{\otimes l^\prime},\g^{\otimes m^\prime})[-m^\prime n-l^\prime]$ 
can be equivalently presented as $n$-shifted polyvectors, which we expand 
in a choice of basis $\{x_a\in\g\}$ and its dual basis $\{\theta^a\in\g^\ast\}$ as 
\begin{subequations}
\begin{flalign}
P \,&=\,P^{a_1\cdots a_m}_{b_1\cdots b_l}~s^nx_{a_1}\cdots s^nx_{a_m}\,s\theta^{b_1}\cdots s\theta^{b_l}\in\Pol^{m,l}\big(\Sym\big(\g^\ast[-1]\big),n\big)\quad,\\
P^\prime \,&=\,P^{\prime c_1\cdots c_{m^\prime}}_{d_1\cdots d_{l^\prime}}~s^n x_{c_1}\cdots s^nx_{c_{m^\prime}}\,s\theta^{d_1}\cdots s\theta^{d_{l^\prime}}\in\Pol^{m^\prime,l^\prime}\big(\Sym\big(\g^\ast[-1]\big),n\big)\quad,
\end{flalign}
\end{subequations}
where we use Einstein's sum convention to suppress summation symbols.
Using the derivation properties \eqref{eqn:compositionder} and definition on generators
\eqref{eqn:compositiongenerators} of the composition operation $\bullet$, one finds
\begin{multline}
(s^{-n-1}P)\bullet(s^{n-1}P^\prime)\,=\, l\,m^\prime\,(-1)^{\vert \theta^{b_l}\vert} ~(-1)^{\sum_{k=2}^{m^\prime}(\vert x_{c_k}\vert+n)\,\sum_{j=1}^{l-1}(\vert\theta^{b_j}\vert+1)}~
P^{a_1\cdots a_m}_{b_1\cdots b_l}~P^{\prime c_1\cdots c_{m^\prime}}_{d_1\cdots d_{l^\prime}}~\du{\theta^{b_l}}{x_{c_1}}\\
\times~s^{-n-1} s^nx_{a_1}\cdots s^nx_{a_m}\,s^{n}x_{c_2}\cdots s^nx_{c_{m^\prime}}\,s\theta^{b_1}\cdots s\theta^{b_{l-1}}\,s\theta^{d_1}\cdots s\theta^{d_{l^\prime}}\quad.
\end{multline}
Transferring this result through the cochain isomorphisms \eqref{eqn:polyvectorleftshift} and \eqref{eqn:polyvector2map}
yields the element-wise formula \eqref{eqn:compositiononmapsexplicit}, expressed in the chosen of basis, 
which is equivalent to the graphical expression \eqref{eqn:compositiononmaps}.
\end{proof}

Combining this result with Corollary \ref{cor:Poissonviabullet}, we obtain
the following graphical characterization and visualization of 
the Maurer-Cartan identities from Remark \ref{rem:tower} for $n$-shifted Poisson structures on a free CDGA.
\begin{cor}\label{cor:shiftedPoissonfree}
The datum of an $n$-shifted Poisson structure on the free CDGA $\Sym\big(\g^\ast[-1]\big)$ is 
equivalent to a family of maps
\begin{flalign}
\left\{~\parbox{1.2cm}{\begin{tikzpicture}
\draw[thick] (-0.25,0.75) -- (-0.25,0);
\draw[thick] (0,0.75) -- (0,0);
\draw[thick] (0.25,0.75) -- (0.25,0);
\draw[thick] (-0.375,0) -- (-0.375,-0.75);
\draw[thick] (-0.125,0) -- (-0.125,-0.75);
\draw[thick] (0.125,0) -- (0.125,-0.75);
\draw[thick] (0.375,0) -- (0.375,-0.75);
\draw[thick,fill=white] (-0.6,-0.25) rectangle (0.6,0.25) node[pos=.5] {\text{\footnotesize{$\pi^{(m,l)}$}}};
\end{tikzpicture}}~\in~ \hom^-_{\pm}\big(\g^{\otimes l},\g^{\otimes m}\big)^{(1-m)n+2-l}~~:~~m\geq 2~,~~l\geq 0\right\}\quad,
\end{flalign}
which, for every fixed $m\geq 2$, is bounded in $l$ and satisfies the following identities
\begin{multline}\label{eqn:shiftedPoissonfree}
(-1)^{(m-1)n+l-1}~\partial ~\parbox{1.2cm}{\begin{tikzpicture}
\draw[thick] (-0.25,0.75) -- (-0.25,0);
\draw[thick] (0,0.75) -- (0,0);
\draw[thick] (0.25,0.75) -- (0.25,0);
\draw[thick] (-0.375,0) -- (-0.375,-0.75);
\draw[thick] (-0.125,0) -- (-0.125,-0.75);
\draw[thick] (0.125,0) -- (0.125,-0.75);
\draw[thick] (0.375,0) -- (0.375,-0.75);
\draw[thick,fill=white] (-0.6,-0.25) rectangle (0.6,0.25) node[pos=.5] {$\pi^{\text{\tiny{$(m,l)$}}}$};
\end{tikzpicture}}
~+\sum_{m_1+m_2-1=m}~\sum_{l_1+l_2-1=l}~(-1)^{((1-m_1)n +2-l_1) \,((m_2-1)n+l_2-1)}\\
\times~(-1)^{n(m_2-1)(l_1-1)}~\sum_{\substack{ \sigma\in\mathrm{Sh}(l_1-1,l_2) \\ \rho\in\mathrm{Sh}(m_1,m_2-1)}} 
(-1)^{\vert \sigma\vert}\,(- 1)^{\vert \rho\vert\,n}
~\parbox{2cm}{\begin{tikzpicture}[scale=0.8]
\draw[thick] (-0.25,1.75) -- (-0.25,0);
\draw[thick] (0,1.75) -- (0,0);
\draw[thick] (0.25,1.75) -- (0.25,0);
\draw[thick] (-1,1.75) -- (-1,-1.25);
\draw[thick] (-1.25,1.75) -- (-1.25,-1.25);
\draw[thick] (-0.375,-0.25) -- (-0.75,-1);
\draw[thick] (-0.125,0) -- (-0.125,-3);
\draw[thick] (0.125,0) -- (0.125,-3);
\draw[thick] (0.375,0) -- (0.375,-3);
\draw[thick] (-1,-1.5) -- (-1,-3);
\draw[thick] (-1.25,-1.5) -- (-1.25,-3);
\draw[thick] (-0.75,-1.5) -- (-0.75,-3);
\draw[thick,fill=white] (-0.7,-0.25) rectangle (0.7,0.25) node[pos=.5] {\text{\footnotesize{$\pi^{(m_2,l_2)}$}}};
\draw[thick,fill=white] (-1.7,-1.5) rectangle (-0.3,-1) node[pos=.5] {\text{\footnotesize{$\pi^{(m_1,l_1)}$}}};
\draw[thick,fill=white,dotted] (-1.5,0.75) rectangle (0.5,1.25) node[pos=.5] {\text{\footnotesize{$\gamma_\sigma$}}};
\draw[thick,fill=white,dotted] (-1.5,-2.5) rectangle (0.5,-2) node[pos=.5] {\text{\footnotesize{$\gamma_\rho$}}};
\end{tikzpicture}}~~=~~0\qquad,
\end{multline}
for all $m\geq 2$ and $l\geq 0$.
\end{cor}

\paragraph{Semi-free CDGAs:} For any CDGA $A$, there exists, because of \eqref{eqn:Schoutenweight<1},
an identification between degree $1$ derivations $\alpha\in \bbT_{A}^1$ which satisfy
the Maurer-Cartan equation in the DGLA of derivations \eqref{eqn:semifreeMC},
and degree $n+2$ and weight $1$ completed $n$-shifted polyvectors
$\widetilde{\alpha}\in \widehat{\Pol}(A,n)^{n+2}$
which satisfy the Maurer-Cartan equation $\partial(s^{-n-1}\widetilde{\alpha})+\frac{1}{2}
\,\pb{s^{-n-1}\widetilde{\alpha}}{s^{-n-1}\widetilde{\alpha}} = 0$ in $\widehat{\Pol}(A,n)$,
i.e.\ with respect to the Schouten-Nijenhuis bracket. This identification is 
given by $\alpha \mapsto\widetilde{\alpha}= s^{n+1}\alpha$.
\sk

For any CDGA $A$ and Maurer-Cartan element $\alpha\in\bbT_A^1$, 
the $\mathbb{P}_{n+2}$-algebra of completed $n$-shifted polyvectors on the deformed CDGA 
$A_\alpha$ with differential $\dd_\alpha = \dd+\alpha$, see \eqref{eqn:deformeddifferential}, 
can be identified with the deformation
\begin{subequations}
\begin{flalign}
\widehat{\Pol}(A_\alpha,n)\,=\,\widehat{\Pol}(A,n)_\alpha
\end{flalign}
of the $\mathbb{P}_{n+2}$-algebra of
completed $n$-shifted polyvectors on $A$ which is 
given by modifying the differential on $\widehat{\Pol}(A,n)$ according to
\begin{flalign}
\partial_\alpha\,:=\, \partial + (-1)^{n+1}\,s^{n+1} \pb{s^{-n-1}s^{n+1}\alpha}{s^{-n-1}(\,\cdot\,)}\quad.
\end{flalign}
\end{subequations}
The Maurer-Cartan equation for $n$-shifted Poisson structures
$\pi =\sum_{m \geq 2}\pi^{(m)}\in\widehat{\Pol}(A_\alpha,n)^{n+2}$ on $A_\alpha$ 
from Definition \ref{def:shiftedPoisson} can then be rewritten as follows
\begin{flalign}
\nn 0\,&=\,\partial_\alpha \big(s^{-n-1}\pi\big) +\tfrac{1}{2} \pb{s^{-n-1}\pi}{s^{-n-1}\pi}\\
\nn \,&=\,\partial \big(s^{-n-1}\pi\big)+\pb{s^{-n-1}s^{n+1}\alpha}{s^{-n-1}\pi}+
\tfrac{1}{2} \pb{s^{-n-1}\pi}{s^{-n-1}\pi}\\
\nn \,&=\,\partial \big(s^{-n-1}\pi\big) 
- \tfrac{1}{2} \pb{s^{-n-1}s^{n+1}\alpha}{s^{-n-1}s^{n+1}\alpha}
+ \tfrac{1}{2} \pbbig{s^{-n-1}\big(s^{n+1}\alpha+\pi\big)}{s^{-n-1}\big(s^{n+1}\alpha+\pi\big)}\\
\,&=\,\partial \big(s^{-n-1}\big(s^{n+1}\alpha+ \pi\big)\big) + \tfrac{1}{2} \pbbig{s^{-n-1}\big(s^{n+1}\alpha+\pi\big)}{s^{-n-1}\big(s^{n+1}\alpha+\pi\big)}\quad,
\end{flalign}
where in the last step we used the Maurer-Cartan equation for $\alpha$.
This means that an $n$-shifted Poisson structure $\pi$ on the 
deformed CDGA $A_\alpha$ is equivalent to a degree $n+2$ and weight $\geq 1$
completed $n$-shifted polyvector $s^{n+1}\alpha + \pi$, whose weight $1$ component
is dictated by $\alpha$, and which satisfies
the Maurer-Cartan equation in the undeformed $n$-shifted polyvectors $\widehat{\Pol}(A,n)$.
Specializing to semi-free CDGAs, we obtain the following generalization
of Corollary \ref{cor:shiftedPoissonfree}.
\begin{cor}\label{cor:shiftedPoissonsemifree}
The combined datum of a semi-free deformation of the free CDGA $\Sym\big(\g^\ast[-1]\big)$
and an $n$-shifted Poisson structure on the associated semi-free CDGA is 
equivalent to a family of maps (including weight $m=1$ components)
\begin{flalign}
\left\{~\parbox{1.2cm}{\begin{tikzpicture}
\draw[thick] (-0.25,0.75) -- (-0.25,0);
\draw[thick] (0,0.75) -- (0,0);
\draw[thick] (0.25,0.75) -- (0.25,0);
\draw[thick] (-0.375,0) -- (-0.375,-0.75);
\draw[thick] (-0.125,0) -- (-0.125,-0.75);
\draw[thick] (0.125,0) -- (0.125,-0.75);
\draw[thick] (0.375,0) -- (0.375,-0.75);
\draw[thick,fill=white] (-0.6,-0.25) rectangle (0.6,0.25) node[pos=.5] {\text{\footnotesize{$\pi^{(m,l)}$}}};
\end{tikzpicture}}~\in~ \hom^-_{\pm}\big(\g^{\otimes l},\g^{\otimes m}\big)^{(1-m)n+2-l}~~:~~m\geq 1~,~~l\geq 0\right\}\quad,
\end{flalign}
which, for every fixed $m\geq 1$, is bounded in $l$ and satisfies the following identities
\begin{multline}\label{eqn:semifreeidentities}
(-1)^{(m-1)n+l-1}~\partial ~\parbox{1.2cm}{\begin{tikzpicture}
\draw[thick] (-0.25,0.75) -- (-0.25,0);
\draw[thick] (0,0.75) -- (0,0);
\draw[thick] (0.25,0.75) -- (0.25,0);
\draw[thick] (-0.375,0) -- (-0.375,-0.75);
\draw[thick] (-0.125,0) -- (-0.125,-0.75);
\draw[thick] (0.125,0) -- (0.125,-0.75);
\draw[thick] (0.375,0) -- (0.375,-0.75);
\draw[thick,fill=white] (-0.6,-0.25) rectangle (0.6,0.25) node[pos=.5] {$\pi^{\text{\tiny{$(m,l)$}}}$};
\end{tikzpicture}}
~+\sum_{m_1+m_2-1=m}~\sum_{l_1+l_2-1=l}~(-1)^{((1-m_1)n +2-l_1) \,((m_2-1)n+l_2-1)}\\
\times~(-1)^{n(m_2-1)(l_1-1)}~\sum_{\substack{ \sigma\in\mathrm{Sh}(l_1-1,l_2) \\ \rho\in\mathrm{Sh}(m_1,m_2-1)}}
(-1)^{\vert \sigma\vert}\,(- 1)^{\vert \rho\vert\,n}
~\parbox{2cm}{\begin{tikzpicture}[scale=0.8]
\draw[thick] (-0.25,1.75) -- (-0.25,0);
\draw[thick] (0,1.75) -- (0,0);
\draw[thick] (0.25,1.75) -- (0.25,0);
\draw[thick] (-1,1.75) -- (-1,-1.25);
\draw[thick] (-1.25,1.75) -- (-1.25,-1.25);
\draw[thick] (-0.375,-0.25) -- (-0.75,-1);
\draw[thick] (-0.125,0) -- (-0.125,-3);
\draw[thick] (0.125,0) -- (0.125,-3);
\draw[thick] (0.375,0) -- (0.375,-3);
\draw[thick] (-1,-1.5) -- (-1,-3);
\draw[thick] (-1.25,-1.5) -- (-1.25,-3);
\draw[thick] (-0.75,-1.5) -- (-0.75,-3);
\draw[thick,fill=white] (-0.7,-0.25) rectangle (0.7,0.25) node[pos=.5] {\text{\footnotesize{$\pi^{(m_2,l_2)}$}}};
\draw[thick,fill=white] (-1.7,-1.5) rectangle (-0.3,-1) node[pos=.5] {\text{\footnotesize{$\pi^{(m_1,l_1)}$}}};
\draw[thick,fill=white,dotted] (-1.5,0.75) rectangle (0.5,1.25) node[pos=.5] {\text{\footnotesize{$\gamma_\sigma$}}};
\draw[thick,fill=white,dotted] (-1.5,-2.5) rectangle (0.5,-2) node[pos=.5] {\text{\footnotesize{$\gamma_\rho$}}};
\end{tikzpicture}}~~=~~0\qquad,
\end{multline}
for all $m\geq 1$ and $l\geq 0$.
\end{cor}

\begin{rem}\label{rem:comparison}
We would like to observe that, in the case where 
$\g\in\Ch$ is a non-positively graded $N$-term cochain complex as in \eqref{eqn:Nterm} below, 
our combined datum from Corollary \ref{cor:shiftedPoissonsemifree} 
of a semi-free deformation of the free CDGA $\Sym\big(\g^\ast[-1]\big)$
and an $n$-shifted Poisson structure on the associated semi-free CDGA 
is equivalent to that of an $L_\infty[0,n-1]$-quasi-bialgebra structure in 
the sense of \cite[Definition 2.5]{Zhu}. In our notation,
these authors introduce, for any non-positively graded
vector space $\g$ and any integers $c,d\in\bbZ$,
the completed symmetric algebra
\begin{flalign}\label{eqn:Zhu}
\prod_{m,l\geq 0}\Sym^m\big(\g[-1-d]\big)\otimes \Sym^{l}\big(\g^{\ast}[-1-c]\big)
\end{flalign}
and endow it with the so-called `big bracket' $\langle\!\langle\,\cdot\,,\,
\cdot\,\rangle\!\rangle$ which is given by extending the duality pairing
$\langle\,\cdot\,,\,\cdot\,\rangle : \g^\ast\otimes \g\to\bbK$ to a biderivation.
They then define an $L_\infty[c,d]$-quasi-bialgebra structure on $\g$
in terms of a degree $3+c+d$ element $t = \sum_{m\geq 1}\sum_{l\geq 0} t^{(m,l)}$
satisfying $\langle\!\langle t,t\rangle\!\rangle=0$.
Comparing \eqref{eqn:Zhu} with \eqref{eqn:shiftedpolyg} and \eqref{eqn:shiftedpolygcomplete},
we observe that, up to a completion in $l$ which is inessential
for $N$-term cochain complexes as in \eqref{eqn:Nterm}, the derived geometric concept
of $n$-shifted polyvectors is given by choosing the integers $(c,d)= (0,n-1)$ in \eqref{eqn:Zhu}.
Furthermore, from \eqref{eqn:Schouteng} we see that the Schouten–Nijenhuis bracket $\pb{\,\cdot\,}{\,\cdot\,}$ 
agrees with the `big bracket' $\langle\!\langle\,\cdot\,,\,\cdot\,\rangle\!\rangle$ because both are given
on the generators by the duality pairing. To relate the identity
$\langle\!\langle t,t\rangle\!\rangle=0$ with the Maurer-Cartan equation in Definition \ref{def:shiftedPoisson},
note that \cite[Definition 2.5]{Zhu} encodes the differential $\dd$ on $\g$ in terms of the $t^{(1,1)}$-component.
Splitting off the differential, one can write $t = \dd + t^\prime$ and one observes
that $\langle\!\langle t,t\rangle\!\rangle=0$ is equivalent to 
the Maurer-Cartan equation for the remainder $t^\prime$. 
\end{rem}


\section{\label{sec:higherLie}Application to higher Chevalley-Eilenberg algebras}
In this section we consider the case where 
\begin{flalign}\label{eqn:Nterm}
\g\,:=\,\Big(
\xymatrix{
\g^{-N+1}\ar[r]^-{\dd}\,&\, \g^{-N+2} \ar[r]^-{\dd}\,&\,\cdots \ar[r]^-{\dd}\,&\,\g^{-1} \ar[r]^-{\dd} \,&\,\g^0
}\Big)\,\in\,\Ch
\end{flalign}
is an $N$-term cochain complex of finite-dimensional vector spaces
which is concentrated in non-positive degrees, where $N\in\bbZ^{\geq 1}$ is a positive integer.
Let us recall that a \textit{Lie $N$-algebra} is a pair $(\g,\ell)$ consisting
of an $N$-term cochain complex as in \eqref{eqn:Nterm} and a family 
$\ell = \big\{\ell_l\in\hom^-(\g^{\otimes l},\g)^{2-l}\big\}_{l\geq 2}$ 
of totally antisymmetric linear maps of degree $2-l$ which satisfy the homotopy Jacobi identities,
see e.g.\ \cite{KraftSchnitzer} and Corollary \ref{cor:Linftystructures} below.
In other words, a Lie $N$-algebra is precisely an $L_\infty$-algebra whose underlying cochain
complex is of the $N$-term form \eqref{eqn:Nterm}. By a simple degree counting argument,
one observes that $\ell_l=0$ must necessarily vanish for all $l> 1+N$,
hence the family of $L_\infty$-brackets $\ell = \{\ell_l\}_{l\geq2}$
is bounded for every  Lie $N$-algebra.
\sk

It is well-known that $L_\infty$-algebra structures on $\g$ 
correspond to (a subclass of) the semi-free deformations of the free CDGA $\Sym\big(\g^\ast[-1]\big)$,
see e.g.\ \cite{KraftSchnitzer}. We recover this result as a special case of 
our graphical calculus from Corollary \ref{cor:shiftedPoissonsemifree}.
\begin{cor}\label{cor:Linftystructures}
Let $\g\in\Ch$ be an $N$-term cochain complex as in \eqref{eqn:Nterm}
and consider a family of maps $\{\pi^{(m,l)}\}_{m\geq 1,l\geq 0}$ as in Corollary 
\ref{cor:shiftedPoissonsemifree} with $\pi^{(1,0)}=0$, $\pi^{(1,1)}=0$ and 
$\pi^{(m,l)} = 0$, for all $m\geq 2$ and $l\geq 0$. Then the identities
\eqref{eqn:semifreeidentities} are equivalent to the $L_\infty$-algebra structure identities
in the sign conventions of \cite[Remark 3.6]{KraftSchnitzer} 
for the family of $L_\infty$-brackets  $\{\ell_l :=(-1)^{l-1} \, \pi^{(1,l)}\}_{l\geq 2}$.
\end{cor}
\begin{proof}
Let us first observe that the non-vanishing maps 
$\pi^{(1,l)}\in \hom_{\pm}^-(\g^{\otimes l},\g)^{2-l}$, for $l\geq 2$, 
have the correct degrees and antisymmetry properties required for an $L_\infty$-algebra structure.
Using that $\pi^{(m,l)} = 0$, for all $m\geq 2$ and $l\geq 0$, the identities in \eqref{eqn:semifreeidentities}
simplify to
\begin{flalign}\label{eqn:tmpLinfty}
(-1)^{l-1}~\partial ~\parbox{1.2cm}{\begin{tikzpicture}
\draw[thick] (-0.25,0.75) -- (-0.25,0);
\draw[thick] (0,0.75) -- (0,0);
\draw[thick] (0.25,0.75) -- (0.25,0);
\draw[thick] (0,0) -- (0,-0.75);
\draw[thick,fill=white] (-0.6,-0.25) rectangle (0.6,0.25) node[pos=.5] {$\pi^{\text{\tiny{$(1,l)$}}}$};
\end{tikzpicture}}
~+\sum_{l_1+l_2-1=l}~(-1)^{(2-l_1) \,(l_2-1)}~\sum_{\sigma\in\mathrm{Sh}(l_1-1,l_2)} (-1)^{\vert \sigma\vert}
~\parbox{2cm}{\begin{tikzpicture}[scale=0.8]
\draw[thick] (-0.25,1.75) -- (-0.25,0);
\draw[thick] (0,1.75) -- (0,0);
\draw[thick] (0.25,1.75) -- (0.25,0);
\draw[thick] (-1,1.75) -- (-1,-1.25);
\draw[thick] (-1.25,1.75) -- (-1.25,-1.25);
\draw[thick] (0,-0.25) -- (-0.75,-1);
\draw[thick] (-1,-1.5) -- (-1,-2);
\draw[thick,fill=white] (-0.7,-0.25) rectangle (0.7,0.25) node[pos=.5] {\text{\footnotesize{$\pi^{(1,l_2)}$}}};
\draw[thick,fill=white] (-1.7,-1.5) rectangle (-0.3,-1) node[pos=.5] {\text{\footnotesize{$\pi^{(1,l_1)}$}}};
\draw[thick,fill=white,dotted] (-1.5,0.75) rectangle (0.5,1.25) node[pos=.5] {\text{\footnotesize{$\gamma_\sigma$}}};
\end{tikzpicture}}~~=~~0\quad,
\end{flalign}
for all $l\geq 2$. Using the antisymmetry properties
\eqref{eqn:input/outputpermutations}, we can permute 
the inputs of $\pi^{(1,l_1)}$ and find
\begin{flalign}
\parbox{2cm}{\begin{tikzpicture}[scale=0.8]
\draw[thick] (-0.25,1.75) -- (-0.25,0);
\draw[thick] (0,1.75) -- (0,0);
\draw[thick] (0.25,1.75) -- (0.25,0);
\draw[thick] (-1,1.75) -- (-1,-1.25);
\draw[thick] (-1.25,1.75) -- (-1.25,-1.25);
\draw[thick] (0,-0.25) -- (-0.75,-1);
\draw[thick] (-1,-1.5) -- (-1,-2);
\draw[thick,fill=white] (-0.7,-0.25) rectangle (0.7,0.25) node[pos=.5] {\text{\footnotesize{$\pi^{(1,l_2)}$}}};
\draw[thick,fill=white] (-1.7,-1.5) rectangle (-0.3,-1) node[pos=.5] {\text{\footnotesize{$\pi^{(1,l_1)}$}}};
\draw[thick,fill=white,dotted] (-1.5,0.75) rectangle (0.5,1.25) node[pos=.5] {\text{\footnotesize{$\gamma_\sigma$}}};
\end{tikzpicture}}~~=~ 
(-1)^{l_1-1}~~
\parbox{2cm}{\begin{tikzpicture}[scale=0.8]
\draw[thick] (0.25,0.25)  -- (1.25,1);
\draw[thick] (1.25,1.5) -- (1.25,2);
\draw[thick] (0,0.25) -- (1,1);
\draw[thick] (1,1.5) -- (1,2);
\draw[thick] (-0.25,0.25)  -- (0.75,1);
\draw[thick] (0.75,1.5) -- (0.75,2);
\draw[thick]  (-0.125,1)-- (1,0.25) -- (1,-1.25);
\draw[thick]  (-0.125,1.5)-- (-0.125,2);
\draw[thick] (0.125,1)-- (1.25,0.25) -- (1.25,-1.25);
\draw[thick] (0.125,1.5)-- (0.125,2);
\draw[thick] (0,-0.25) -- (0.75,-1);
\draw[thick] (1,-1.5) -- (1,-2);
\draw[thick,fill=white] (-0.7,-0.25) rectangle (0.7,0.25) node[pos=.5] {\text{\footnotesize{$\pi^{(1,l_2)}$}}};
\draw[thick,fill=white] (1.7,-1.5) rectangle (0.3,-1) node[pos=.5] {\text{\footnotesize{$\pi^{(1,l_1)}$}}};
\draw[thick,fill=white,dotted] (-0.5,1) rectangle (1.5,1.5) node[pos=.5] {\text{\footnotesize{$\gamma_\sigma$}}};
\end{tikzpicture}}\quad.
\end{flalign}
Inserting this back into \eqref{eqn:tmpLinfty} and using also that the sum over shuffles yields
an antisymmetric map, we obtain
\begin{flalign}
(-1)^{l-1}~\partial ~\parbox{1.2cm}{\begin{tikzpicture}
\draw[thick] (-0.25,0.75) -- (-0.25,0);
\draw[thick] (0,0.75) -- (0,0);
\draw[thick] (0.25,0.75) -- (0.25,0);
\draw[thick] (0,0) -- (0,-0.75);
\draw[thick,fill=white] (-0.6,-0.25) rectangle (0.6,0.25) node[pos=.5] {$\pi^{\text{\tiny{$(1,l)$}}}$};
\end{tikzpicture}}
~+\sum_{l_1+l_2-1=l}~(-1)^{l_2-1}~\sum_{\sigma\in\mathrm{Sh}(l_2,l_1-1)} (-1)^{\vert \sigma\vert}~~
~\parbox{2cm}{\begin{tikzpicture}[scale=0.8]
\draw[thick] (0.25,0.25)  -- (0.25,1.75);
\draw[thick] (0,0.25) -- (0,1.75);
\draw[thick] (-0.25,0.25)  -- (-0.25,1.75);
\draw[thick]  (1,1.75)-- (1,0.25) -- (1,-1.25);
\draw[thick] (1.25,1.75)-- (1.25,0.25) -- (1.25,-1.25);
\draw[thick] (0,-0.25) -- (0.75,-1);
\draw[thick] (1,-1.5) -- (1,-2);
\draw[thick,fill=white] (-0.7,-0.25) rectangle (0.7,0.25) node[pos=.5] {\text{\footnotesize{$\pi^{(1,l_2)}$}}};
\draw[thick,fill=white] (1.7,-1.5) rectangle (0.3,-1) node[pos=.5] {\text{\footnotesize{$\pi^{(1,l_1)}$}}};
\draw[thick,fill=white,dotted] (-0.5,0.75) rectangle (1.5,1.25) node[pos=.5] {\text{\footnotesize{$\gamma_\sigma$}}};
\end{tikzpicture}}~~~=~~0\quad,
\end{flalign}
for all $l\geq 2$, where the sum in this expression is over $(l_2,l_1{-}1)$-shuffles, 
in contrast to the $(l_1{-}1,l_2)$-shuffles in \eqref{eqn:tmpLinfty}.
Under the identification $\ell_l :=(-1)^{l-1} \, \pi^{(1,l)}$, for all $l\geq 2$, this
gives the $L_\infty$-algebra structure identities in the sign conventions
of \cite[Remark 3.6]{KraftSchnitzer}.
\end{proof}

\begin{rem}
Dropping in Corollary \ref{cor:Linftystructures} 
the requirement that $\pi^{(1,0)} = 0 $ and $\pi^{(1,1)}=0$ leads to more general
semi-free deformations of the CDGA $\Sym\big(\g^\ast[-1]\big)$. 
The component $\pi^{(1,1)}$ induces a deformation of the given
differential $\dd$ of $\g\in\Ch$ and the component $\pi^{(1,0)}$ introduces a violation
of the square-zero condition of the differential, which is called
curvature and leads to the concept of curved $L_\infty$-algebras. 
We do not consider these more general deformations in this work and restrict
our attention to ordinary $L_\infty$-algebras.
\end{rem}

The \textit{Chevalley-Eilenberg algebra} of a 
Lie $N$-algebra $\g=(\g,\ell)$ is defined as the semi-free deformation 
\begin{flalign}
\CE(\g)\,:=\, \big(\Sym\big(\g^{\ast}[-1]\big)\big)_\ell
\end{flalign} 
of the free CDGA $\Sym\big(\g^\ast[-1]\big)$ 
along the Maurer-Cartan element from Corollary \ref{cor:Linftystructures} 
which is associated to the $L_\infty$-brackets $\ell = \{\ell_l\}_{l\geq 2}$.
One can interpret
this semi-free CDGA as an algebraic model for the formal classifying stack 
$\mathrm{B}\g=[\mathrm{pt}/\g]$ of the Lie $N$-algebra,
see e.g.\ \cite[Example 3.6]{Pridham} and \cite[Examples 1.10]{PridhamOutline}.
Corollary \ref{cor:shiftedPoissonsemifree} then provides a characterization
of the $n$-shifted Poisson structures on $\CE(\g)$ by fixing the weight $m=1$ 
components of the family of maps $\{\pi^{(m,l)}\}_{m\geq 1,l\geq 0}$ in terms of the $L_\infty$-brackets
according to $\pi^{(1,0)}=0$, $\pi^{(1,1)}=0$ and $\pi^{(1,l)} = (-1)^{l-1}\,\ell_l$, for all $l\geq 2$.
The aim of the following two subsections is to work out this characterization explicitly
for positive $n\geq 1$ in the simplest two cases where $\g$ is an ordinary Lie algebra 
and where $\g$ is a Lie $2$-algebra.
In order to facilitate this characterization, let us start with a general observation
about the relationship between the shift $n$ of the shifted Poisson structure and the positive 
integer $N$ of the Lie $N$-algebra.
\begin{lem}\label{lem:degreecounting}
Let $\g=(\g,\ell)$ be any Lie $N$-algebra. Then every $n$-shifted
Poisson structure on $\CE(\g)$ is necessarily trivial for $n>2N$.
\end{lem}
\begin{proof}
We argue by degree counting that every component $\pi^{(m,l)}\in 
\hom^-_{\pm}\big(\g^{\otimes l},\g^{\otimes m}\big)^{(1-m)n+2-l}$,
for $m\geq 2$ and $l\geq 0$, of an $n$-shifted Poisson structure $\pi$ on $\CE(\g)$
is necessarily trivial in the case where $n>2N$. Since $\g$ is by hypothesis
\eqref{eqn:Nterm} concentrated in degrees $\{-N+1,\dots,0\}$, it follows that 
the internal hom complex $\hom\big(\g^{\otimes l},\g^{\otimes m}\big)\cong
\g^{\otimes m}\otimes {\g^\ast}^{\otimes l}$ is concentrated
in degrees $\{m(-N+1),\dots, l(N-1)\}$. Hence, for the existence
of a non-trivial element $\pi^{(m,l)}$ of degree $(1-m)n+2-l$ the bounds
\begin{flalign}\label{eqn:inequality}
m(-N+1)\,\leq\, (1-m)n+2-l\,\leq\, l(N-1)
\end{flalign}
must be satisfied. 
Since $m\geq 2$, we have that $1-m<0$
and hence the hypothesis $n>2N$ of this lemma yields the inequality
\begin{flalign}
(1-m)n+2-l\,<\, (1-m)2N +2-l\,\leq\, (1-m)2N +2\quad,
\end{flalign}
where in the last step we also used that $l\geq 0$. Combining this inequality 
with the first inequality in \eqref{eqn:inequality} we find
\begin{flalign}
m\,(-N+1)\,<\,(1-m)2N +2\quad \Longleftrightarrow\quad m(N+1)\,<\, 2(N+1)\quad 
\Longleftrightarrow\quad m\,< \,2\quad.
\end{flalign}
This is inconsistent with the fact that every shifted Poisson structure has only 
components with $m\geq 2$, which completes our proof.
\end{proof}

\begin{conv}
For better readability of the graphical identities 
in the following subsections, we will adopt the compact notation
\begin{subequations}\label{eqn:shuffle-sums-suppressed}
\begin{flalign}\label{eqn:shuffle-sums-suppressed-even}
\parbox{1.25cm}{\begin{tikzpicture}
\draw[thick] (-0.6,-0.6) -- (-0.65,-1.7);
\draw[thick] (-0.2,-0.6) -- (-0.2,-1.7);
\draw[thick] (0.2,-0.6) -- (0.2,-1.7);
\draw[thick] (0.6,-0.6) -- (0.65,-1.7);
\draw[thick,fill=white,dotted] (-0.8,-1.4) rectangle (0.8,-0.9) node[pos=0.5] {$\scriptstyle(k_1,k_2)^\even$};
\draw[thick, decorate,decoration={brace,amplitude=5pt}]
  (-0.65,-0.5) -- (-0.15,-0.5) node[midway,yshift=1em]{$\scriptstyle k_1$};
\draw[thick, decorate,decoration={brace,amplitude=5pt}]
  (0.15,-0.5) -- (0.65,-0.5) node[midway,yshift=1em]{$\scriptstyle k_2$};
\end{tikzpicture}}
~~~~~&:=~\sum_{\tau\in\mathrm{Sh}(k_1,k_2)\!\!}~~\parbox{1.25cm}{\begin{tikzpicture}
\draw[thick] (-0.375,-0.6) -- (-0.375,-1.7);
\draw[thick] (-0.125,-0.6) -- (-0.125,-1.7);
\draw[thick] (0.125,-0.6) -- (0.125,-1.7);
\draw[thick] (0.375,-0.6) -- (0.375,-1.7);
\draw[thick,fill=white,dotted] (-0.6,-1.4) rectangle (0.6,-0.9) node[pos=0.5] {$\gamma_\tau$};
\end{tikzpicture}}\qquad , 
\\[5pt]\label{eqn:shuffle-sums-suppressed-odd}
\parbox{1.25cm}{\begin{tikzpicture}
\draw[thick] (-0.6,-0.6) -- (-0.65,-1.7);
\draw[thick] (-0.2,-0.6) -- (-0.2,-1.7);
\draw[thick] (0.2,-0.6) -- (0.2,-1.7);
\draw[thick] (0.6,-0.6) -- (0.65,-1.7);
\draw[thick,fill=white,dotted] (-0.8,-1.4) rectangle (0.8,-0.9) node[pos=0.5] {$\scriptstyle(k_1,k_2)^\odd$};
\draw[thick, decorate,decoration={brace,amplitude=5pt}]
  (-0.65,-0.5) -- (-0.15,-0.5) node[midway,yshift=1em]{$\scriptstyle k_1$};
\draw[thick, decorate,decoration={brace,amplitude=5pt}]
  (0.15,-0.5) -- (0.65,-0.5) node[midway,yshift=1em]{$\scriptstyle k_2$};
\end{tikzpicture}}
~~~~~&:=~
\sum_{\tau\in\mathrm{Sh}(k_1,k_2)\!\!}(-1)^{\vert \tau \vert}~~\parbox{1.25cm}{\begin{tikzpicture}
\draw[thick] (-0.375,-0.6) -- (-0.375,-1.7);
\draw[thick] (-0.125,-0.6) -- (-0.125,-1.7);
\draw[thick] (0.125,-0.6) -- (0.125,-1.7);
\draw[thick] (0.375,-0.6) -- (0.375,-1.7);
\draw[thick,fill=white,dotted] (-0.6,-1.4) rectangle (0.6,-0.9) node[pos=0.5] {$\gamma_\tau$};
\end{tikzpicture}}\qquad ,
\end{flalign}
\end{subequations}
which suppresses the summation over $(k_1,k_2)$-shuffle permutations and their signs. 
Note that in the identities for $n$-shifted Poisson structures in Corollary \ref{cor:shiftedPoissonsemifree},
the shuffles attached to the inputs of maps $\hom_{\pm}^{-}\big(\g^{\otimes l},\g^{\otimes m}\big)$
always carry odd signs as in \eqref{eqn:shuffle-sums-suppressed-odd} 
while the shuffles attached to the outputs carry either odd signs \eqref{eqn:shuffle-sums-suppressed-odd} 
or even signs \eqref{eqn:shuffle-sums-suppressed-even} depending on the parity of $n$.
\end{conv}

\subsection{\label{subsec:Lie}Ordinary Lie algebras}
Suppose that the cochain complex $\g$ is concentrated in degree $0$, which necessarily forces the differential 
$\dd=0$ to be trivial. It follows that the internal homs $\hom(\g^{\otimes m},\g^{\otimes l})$ are concentrated
in degree $0$ too and carry the trivial differential $\partial =0$. This implies that 
the weight $1$ components $\pi^{(1,l)}\in \hom_{\pm}^-(\g^{\otimes l},\g)^{2-l}$
of the family of maps in Corollary \ref{cor:shiftedPoissonsemifree} are necessarily
trivial for $l\neq 2$. We shall visualize the only non-trivial component by
\begin{flalign}\label{eqn:Liebracket}
\pi^{(1,2)}\,=~\parbox{0.4cm}{\begin{tikzpicture}[scale=0.6]
\draw[fill=black] (0,0) circle (2pt);
\draw[thick] (0,0) -- (0,-0.5);
\draw[thick] (-0.25,0.5) -- (0,0);
\draw[thick] (0.25,0.5) -- (0,0);
\end{tikzpicture}}\in ~\hom_{\pm}^{-}\big(\g^{\otimes 2},\g\big)^0\quad.
\end{flalign}
As a consequence of the identities \eqref{eqn:semifreeidentities} for $m=1$ (see also 
Corollary \ref{cor:Linftystructures}), 
it follows that this defines a Lie algebra structure on $\g$, i.e.\ the Jacobi identity
\begin{flalign}
0~=~
\parbox{0.8cm}{\begin{tikzpicture}[scale=0.6]
\draw[fill=black] (0,0) circle (2pt);
\draw[thick] (0,0) -- (0,-0.5);
\draw[thick] (-0.5,1) -- (0,0);
\draw[thick] (0.25,0.5) -- (0,0);
\draw[fill=black] (0.25,0.5) circle (2pt);
\draw[thick]  (0.25,0.5) -- (0,1);
\draw[thick]  (0.25,0.5) -- (0.5,1);
\draw[thick] (-0.5,1) -- (-0.5,2);
\draw[thick] (0,1) -- (0,2);
\draw[thick] (0.5,1) -- (0.5,2);
\draw[thick,fill=white,dotted] (-0.8,0.9) rectangle (0.8,1.6) node[pos=.5] {$\scriptstyle(1,2)^\odd$};
\end{tikzpicture}}~~~=~
\parbox{0.6cm}{\begin{tikzpicture}[scale=0.6]
\draw[fill=black] (0,0) circle (2pt);
\draw[thick] (0,0) -- (0,-0.5);
\draw[thick] (-0.5,1) -- (0,0);
\draw[thick] (0.25,0.5) -- (0,0);
\draw[fill=black] (0.25,0.5) circle (2pt);
\draw[thick]  (0.25,0.5) -- (0,1);
\draw[thick]  (0.25,0.5) -- (0.5,1);
\draw[thick] (-0.5,1) -- (-0.5,1.5);
\draw[thick] (0,1) -- (0,1.5);
\draw[thick] (0.5,1) -- (0.5,1.5);
\end{tikzpicture}}
~~-~~
\parbox{0.6cm}{\begin{tikzpicture}[scale=0.6]
\draw[fill=black] (0,0) circle (2pt);
\draw[thick] (0,0) -- (0,-0.5);
\draw[thick] (-0.5,1) -- (0,0);
\draw[thick] (0.25,0.5) -- (0,0);
\draw[fill=black] (0.25,0.5) circle (2pt);
\draw[thick]  (0.25,0.5) -- (0,1);
\draw[thick]  (0.25,0.5) -- (0.5,1);
\draw[thick] (-0.5,1) -- (0,1.5);
\draw[thick] (0,1) -- (-0.5,1.5);
\draw[thick] (0.5,1) -- (0.5,1.5);
\end{tikzpicture}}
~~+~~
\parbox{0.6cm}{\begin{tikzpicture}[scale=0.6]
\draw[fill=black] (0,0) circle (2pt);
\draw[thick] (0,0) -- (0,-0.5);
\draw[thick] (-0.5,1) -- (0,0);
\draw[thick] (0.25,0.5) -- (0,0);
\draw[fill=black] (0.25,0.5) circle (2pt);
\draw[thick]  (0.25,0.5) -- (0,1);
\draw[thick]  (0.25,0.5) -- (0.5,1);
\draw[thick] (-0.5,1) -- (0.5,1.5);
\draw[thick] (0,1) -- (-0.5,1.5);
\draw[thick] (0.5,1) -- (0,1.5);
\end{tikzpicture}}
\end{flalign}
holds true. (The usual form of the Jacobi identity with cyclic permutations, in contrast to shuffle
permutations, is obtained by using antisymmetry \eqref{eqn:input/outputpermutations} at the 
top vertex of the second term.) 
\sk

Our aim is to characterize the $n$-shifted Poisson structures
on the associated Chevalley-Eilenberg algebra $\CE(\g)$ of this Lie algebra,
for all $n\geq 1$. Using Lemma \ref{lem:degreecounting},
we observe that the $n$-shifted Poisson structures
on the Chevalley-Eilenberg algebra $\CE(\g)$ are trivial for all $n>2$. 
For $n=2$ and $n=1$, we recover from Corollary \ref{cor:shiftedPoissonsemifree}
the results of Safronov \cite[Proposition 2.6 and Theorem 2.8]{Safronov},
which we shall restate in our graphical calculus in the following two propositions.

\begin{propo}\label{prop:2shiftedLie}
A $2$-shifted Poisson structure on the Chevalley-Eilenberg algebra $\CE(\g)$ of an ordinary Lie algebra
$\g$ with bracket \eqref{eqn:Liebracket} is given by the datum of a degree $0$ map
\begin{flalign}
\pi^{(2,0)}\,=~\parbox{0.4cm}{\begin{tikzpicture}[scale=0.6]
\draw[fill=black] (0,0) circle (2pt);
\draw[thick] (0,0) -- (-0.25,-0.5);
\draw[thick] (0,0) -- (0.25,-0.5);
\end{tikzpicture}}\in~ \hom_{+}^-\big(\bbK,\g^{\otimes 2}\big)^0\quad,
\end{flalign}
which satisfies the identity
\begin{flalign}
0~&=~
\parbox{0.8cm}{\begin{tikzpicture}[scale=0.6]
\draw[fill=black] (0,0) circle (2pt);
\draw[thick] (0,0) -- (-0.25,-0.5);
\draw[thick] (0,0) -- (0.5,-1);
\draw[fill=black] (-0.25,-0.5) circle (2pt);
\draw[thick] (-0.75,0.5) -- (-0.25,-0.5) -- (-0.25,-1);
\draw[thick] (-0.25,-1) -- (-0.25,-2);
\draw[thick] (0.5,-1) -- (0.5,-2);
\draw[thick,fill=white,dotted] (-0.7,-0.9) rectangle (1,-1.6) node[pos=.5] {$\scriptstyle (1,1)^\even$};
\end{tikzpicture}}
~~~=~
\parbox{0.8cm}{\begin{tikzpicture}[scale=0.6]
\draw[fill=black] (0,0) circle (2pt);
\draw[thick] (0,0) -- (-0.25,-0.5);
\draw[thick] (0,0) -- (0.5,-1);
\draw[fill=black] (-0.25,-0.5) circle (2pt);
\draw[thick] (-0.75,0.5) -- (-0.25,-0.5) -- (-0.25,-1);
\draw[thick] (-0.25,-1) -- (-0.25,-1.6);
\draw[thick] (0.5,-1) -- (0.5,-1.6);
\end{tikzpicture}}
~+~
\parbox{0.8cm}{\begin{tikzpicture}[scale=0.6]
\draw[fill=black] (0,0) circle (2pt);
\draw[thick] (0,0) -- (-0.25,-0.5);
\draw[thick] (0,0) -- (0.5,-1);
\draw[fill=black] (-0.25,-0.5) circle (2pt);
\draw[thick] (-0.75,0.5) -- (-0.25,-0.5) -- (-0.25,-1);
\draw[thick] (-0.25,-1) -- (0.5,-1.6);
\draw[thick] (0.5,-1) -- (-0.25,-1.6);
\end{tikzpicture}}
~~=~0\quad.
\end{flalign}
This datum is equivalent to an adjoint action invariant symmetric tensor in $\g^{\otimes 2}$, 
i.e., an element in $(\Sym^2\,\g)^\g$.
\end{propo}
\begin{proof}
This result follows directly from Corollary \ref{cor:shiftedPoissonsemifree} and a simple degree 
counting argument. For the latter recall that the cochain complex 
$\hom(\g^{\otimes l},\g^{\otimes m})$ is concentrated in degree $0$,
so for a component $\pi^{(m,l)}\in \hom^-_{+}(\g^{\otimes l},\g^{\otimes m})^{(1-m)2+2-l}$
to be non-trivial its degree must satisfy $(1-m)2+2-l = 4-2m -l=0$. For $m\geq 2$ and $l\geq 0$, 
this is only the case for $\pi^{(2,0)}$. The adjoint action invariance condition follows then from
the $(m,l) = (2,1)$ component of \eqref{eqn:semifreeidentities}.
\end{proof}

\begin{propo}\label{prop:1shiftedLie}
A $1$-shifted Poisson structure on the Chevalley-Eilenberg algebra $\CE(\g)$ of an ordinary Lie algebra
$\g$ with bracket \eqref{eqn:Liebracket} is given by the datum of two degree $0$ maps
\begin{flalign}
\pi^{(2,1)}\,=~\parbox{0.4cm}{\begin{tikzpicture}[scale=0.6]
\draw[fill=black] (0,0) circle (2pt);
\draw[thick] (0,0) -- (-0.25,-0.5);
\draw[thick] (0,0) -- (0.25,-0.5);
\draw[thick] (0,0.5) -- (0,0);
\end{tikzpicture}}\in~ \hom_{-}^{-}\big(\g,\g^{\otimes 2}\big)^0\quad,\qquad
\pi^{(3,0)}\,=~\parbox{0.6cm}{\begin{tikzpicture}[scale=0.6]
\draw[fill=black] (0,0) circle (2pt);
\draw[thick] (0,0) -- (-0.35,-0.5);
\draw[thick] (0,0) -- (0,-0.5);
\draw[thick] (0,0) -- (0.35,-0.5);
\end{tikzpicture}}~\in~ \hom_{-}^{-}\big(\bbK,\g^{\otimes 3}\big)^0\quad,
\end{flalign}
which satisfy the identities
\begin{subequations}
\begin{flalign}
0~&=~\parbox{0.4cm}{\begin{tikzpicture}[scale=0.7]
\draw[fill=black] (0,0) circle (2pt);
\draw[thick] (-0.25,0.5) -- (0,0);
\draw[thick] (0.25,0.5) -- (0,0);
\draw[thick] (0,0) -- (0,-0.5);
\draw[fill=black] (0,-0.5) circle (2pt);
\draw[thick] (0,-0.5) -- (-0.25,-1);
\draw[thick] (0,-0.5) -- (0.25,-1);
\end{tikzpicture}}~~ -
\parbox{0.9cm}{\begin{tikzpicture}[scale=0.6]
\draw[fill=black] (0,0) circle (2pt);
\draw[thick] (0,1.5)-- (0,0.5) -- (0,0);
\draw[thick] (0,0) -- (-0.25,-0.5);
\draw[thick] (0,0) -- (0.5,-1) -- (0.5,-2);
\draw[fill=black] (-0.25,-0.5) circle (2pt);
\draw[thick] (-0.25,-0.5) -- (-0.75,0.5) -- (-0.75,1.5);
\draw[thick] (-0.25,-0.5) -- (-0.25,-2) ;
\draw[thick,fill=white,dotted] (-0.7,-0.9) rectangle (0.9,-1.6) node[pos=.5] {$\scriptstyle(1,1)^\odd$};
\draw[thick,fill=white,dotted] (-1.2,0.5) rectangle (0.4,1.2) node[pos=.5] {$\scriptstyle(1,1)^\odd$};
\end{tikzpicture}}~~\quad,\\[5pt]
0~&=~
\parbox{0.8cm}{\begin{tikzpicture}[scale=0.6]
\draw[fill=black] (0,0) circle (2pt);
\draw[thick] (0,0) -- (0,0.5);
\draw[thick] (0.5,-1) -- (0,0);
\draw[thick] (-0.25,-0.5) -- (0,0);
\draw[fill=black] (-0.25,-0.5) circle (2pt);
\draw[thick]  (-0.25,-0.5) -- (0,-1);
\draw[thick]  (-0.25,-0.5) -- (-0.5,-1);
\draw[thick] (0.5,-1) -- (0.5,-2);
\draw[thick] (0,-1) -- (0,-2);
\draw[thick] (-0.5,-1) -- (-0.5,-2);
\draw[thick,fill=white,dotted] (-0.8,-1) rectangle (0.8,-1.7) node[pos=.5] {$\scriptstyle(2,1)^\odd$};
\end{tikzpicture}}~~~+~~
\parbox{0.85cm}{\begin{tikzpicture}[scale=0.6]
\draw[fill=black] (0,0) circle (2pt);
\draw[thick] (0,0) -- (-0.25,-0.5);
\draw[thick] (0,0) -- (0.25,-1) -- (0.25,-2);
\draw[thick] (0,0) -- (0.75,-1) -- (0.75,-2);
\draw[fill=black] (-0.25,-0.5) circle (2pt);
\draw[thick] (-0.25,-0.5) -- (-0.25,-2);
\draw[thick] (-0.25,-0.5) -- (-0.75,0.5);
\draw[thick,fill=white,dotted] (-0.55,-1) rectangle (1.05,-1.7) node[pos=.5] {$\scriptstyle(1,2)^\odd$};
\end{tikzpicture}}~~\quad,\\[5pt]
0 ~&=~\parbox{1cm}{\begin{tikzpicture}[scale=0.6]
\draw[fill=black] (0,0) circle (2pt);
\draw[thick] (0,0) -- (-0.5,-0.5);
\draw[thick] (0,0) -- (0.25,-1) -- (0.25,-2);
\draw[thick] (0,0) -- (0.75,-1) -- (0.75,-2);
\draw[fill=black] (-0.5,-0.5) circle (2pt);
\draw[thick] (-0.5,-0.5) -- (-0.75,-1) -- (-0.75,-2);
\draw[thick] (-0.5,-0.5) -- (-0.25,-1) -- (-0.25,-2);
\draw[thick,fill=white,dotted] (-1,-1) rectangle (1,-1.7) node[pos=.5] {$\scriptstyle(2,2)^\odd$};
\end{tikzpicture}}~~\quad.
\end{flalign}
\end{subequations}
These data are equivalent to a quasi-Lie bialgebra structure on the Lie algebra 
$\g$, see e.g.\ \cite[Section 16.2]{EtingofSchiffmann}. 
\end{propo}
\begin{proof}
The proof is again a direct consequence of Corollary 
\ref{cor:shiftedPoissonsemifree} and a simple degree counting argument, hence 
we do not have to spell out the details.
\end{proof}

\subsection{\label{subsec:2Lie}Lie $2$-algebras}
Consider a $2$-term cochain complex
\begin{flalign}\label{eqn:2term}
\g\,:=\,\Big(
\xymatrix{
\g^{-1} \ar[r]^-{\dd} \,&\,\g^0
}\Big)\,\in\,\Ch\quad.
\end{flalign}
Applying Corollary \ref{cor:Linftystructures} and a simple degree counting argument,
one observes that an $L_\infty$-algebra structure on $\g$ consists of
two maps
\begin{flalign}\label{eqn:Lie2bracket}
\pi^{(1,2)}\,=~\parbox{0.4cm}{\begin{tikzpicture}[scale=0.6]
\draw[fill=black] (0,0) circle (2pt);
\draw[thick] (0,0) -- (0,-0.5);
\draw[thick] (-0.25,0.5) -- (0,0);
\draw[thick] (0.25,0.5) -- (0,0);
\end{tikzpicture}}\in ~\hom_{\pm}^{-}\big(\g^{\otimes 2},\g\big)^0\quad,\qquad
\pi^{(1,3)}\,=~\parbox{0.55cm}{\begin{tikzpicture}[scale=0.6]
\draw[fill=black] (0,0) circle (2pt);
\draw[thick] (0,0) -- (0,-0.5);
\draw[thick] (-0.35,0.5) -- (0,0);
\draw[thick] (0,0.5) -- (0,0);
\draw[thick] (0.35,0.5) -- (0,0);
\end{tikzpicture}}\in ~\hom_{\pm}^{-}\big(\g^{\otimes 3},\g\big)^{-1}
\end{flalign}
of degree $0$ and, respectively, degree $-1$. The $L_\infty$-algebra structure identities
in \eqref{eqn:semifreeidentities} (see also Corollary \ref{cor:Linftystructures}) 
specialize in the present case to the identities
\begin{subequations}\label{eqn:Lie2bracketidentities}
\begin{flalign}
\partial\,\parbox{0.4cm}{\begin{tikzpicture}[scale=0.6]
\draw[fill=black] (0,0) circle (2pt);
\draw[thick] (0,0) -- (0,-0.5);
\draw[thick] (-0.25,0.5) -- (0,0);
\draw[thick] (0.25,0.5) -- (0,0);
\end{tikzpicture}}~&=~0\quad,\\[5pt]
-\partial\,\parbox{0.55cm}{\begin{tikzpicture}[scale=0.6]
\draw[fill=black] (0,0) circle (2pt);
\draw[thick] (0,0) -- (0,-0.5);
\draw[thick] (-0.35,0.5) -- (0,0);
\draw[thick] (0,0.5) -- (0,0);
\draw[thick] (0.35,0.5) -- (0,0);
\end{tikzpicture}}~&=~
\parbox{0.8cm}{\begin{tikzpicture}[scale=0.6]
\draw[fill=black] (0,0) circle (2pt);
\draw[thick] (0,0) -- (0,-0.5);
\draw[thick] (-0.5,1) -- (0,0);
\draw[thick] (0.25,0.5) -- (0,0);
\draw[fill=black] (0.25,0.5) circle (2pt);
\draw[thick]  (0.25,0.5) -- (0,1);
\draw[thick]  (0.25,0.5) -- (0.5,1);
\draw[thick] (-0.5,1) -- (-0.5,2);
\draw[thick] (0,1) -- (0,2);
\draw[thick] (0.5,1) -- (0.5,2);
\draw[thick,fill=white,dotted] (-0.8,0.9) rectangle (0.8,1.6) node[pos=.5] {$\scriptstyle(1,2)^\odd$};
\end{tikzpicture}}~\quad,\\[5pt]
0~&=~-~
\parbox{1.2cm}{\begin{tikzpicture}[scale=0.6]
\draw[fill=black] (0,0) circle (2pt);
\draw[thick] (0,0) -- (0,-0.5);
\draw[thick] (-0.5,1) -- (0,0);
\draw[thick] (0,1) -- (0,0);
\draw[thick] (0.5,0.5) -- (0,0);
\draw[fill=black] (0.5,0.5) circle (2pt);
\draw[thick]  (0.5,0.5) -- (0.25,1);
\draw[thick]  (0.5,0.5) -- (0.75,1);
\draw[thick] (-0.5,1) -- (-0.5,2);
\draw[thick] (0,1) -- (0,2);
\draw[thick] (0.25,1) -- (0.25,2);
\draw[thick] (0.75,1) -- (0.75,2);
\draw[thick,fill=white,dotted] (-0.7,1) rectangle (0.9,1.7) node[pos=.5] {$\scriptstyle(2,2)^\odd$};
\end{tikzpicture}}
+~
\parbox{0.8cm}{\begin{tikzpicture}[scale=0.6]
\draw[fill=black] (0,0) circle (2pt);
\draw[thick] (0,0) -- (0,-0.5);
\draw[thick] (-0.5,1) -- (0,0);
\draw[thick] (0.25,0.5) -- (0,0);
\draw[fill=black] (0.25,0.5) circle (2pt);
\draw[thick]  (0.25,0.5) -- (0,1);
\draw[thick]  (0.25,0.5) -- (0.25,1);
\draw[thick]  (0.25,0.5) -- (0.5,1);
\draw[thick] (-0.5,1) -- (-0.5,2);
\draw[thick] (0,1) -- (0,2);
\draw[thick] (0.25,1) -- (0.25,2);
\draw[thick] (0.5,1) -- (0.5,2);
\draw[thick,fill=white,dotted] (-0.8,1) rectangle (0.8,1.7) node[pos=.5] {$\scriptstyle(1,3)^\odd$};
\end{tikzpicture}}~\quad.
\end{flalign}
\end{subequations}
The first identity expresses that the binary bracket is a cochain map
and the second identity states that the ternary bracket is a homotopy
witnessing the Jacobi identity. The third identity is an algebraic relation
between the binary and ternary brackets. We would like to note that the 
quadratic identity for the ternary bracket
\begin{flalign}\label{eqn:ternarysquare}
0~=~
\parbox{1.2cm}{\begin{tikzpicture}[scale=0.6]
\draw[fill=black] (0,0) circle (2pt);
\draw[thick] (0,0) -- (0,-0.5);
\draw[thick] (-0.5,1) -- (0,0);
\draw[thick] (0,1) -- (0,0);
\draw[thick] (0.5,0.5) -- (0,0);
\draw[fill=black] (0.5,0.5) circle (2pt);
\draw[thick]  (0.5,0.5) -- (0.25,1);
\draw[thick]  (0.5,0.5) -- (0.5,1);
\draw[thick]  (0.5,0.5) -- (0.75,1);
\draw[thick] (-0.5,1) -- (-0.5,2);
\draw[thick] (0,1) -- (0,2);
\draw[thick] (0.25,1) -- (0.25,2);
\draw[thick] (0.5,1) -- (0.5,2);
\draw[thick] (0.75,1) -- (0.75,2);
\draw[thick,fill=white,dotted] (-0.7,1) rectangle (0.9,1.7) node[pos=.5] {$\scriptstyle(2,3)^\odd$};
\end{tikzpicture}}\quad,
\end{flalign}
which arises from setting $(m,l)=(1,5)$ in \eqref{eqn:semifreeidentities},
holds true automatically. This is due to the fact that the only non-trivial component of 
a degree $-1$ map $\pi^{(1,3)}$ on a $2$-term complex \eqref{eqn:2term} is given by 
$\pi^{(1,3)}:\g^0\otimes\g^0\otimes \g^0\to\g^{-1}$, hence this map composes trivially with itself.
\sk

Our aim is to characterize the $n$-shifted Poisson structures on
the associated Chevalley-Eilenberg algebra $\CE(\g)$ of this Lie $2$-algebra,
for all $n\geq 1$. Using Lemma \ref{lem:degreecounting},
we observe that the $n$-shifted Poisson structures
on the Chevalley-Eilenberg algebra $\CE(\g)$ are trivial for all $n>4$. 
Hence, this problem reduces to the cases $n\in\{1,2,3,4\}$,
which can be treated by applying Corollary \ref{cor:shiftedPoissonsemifree}
and performing suitable degree counting arguments.
\begin{propo}\label{prop:4shifted2Lie}
A $4$-shifted Poisson structure on the Chevalley-Eilenberg algebra $\CE(\g)$ of a Lie $2$-algebra
$\g$ with brackets \eqref{eqn:Lie2bracket} is given by the datum of a degree $-2$ map
\begin{flalign}\label{eqn:42Liestructures}
\pi^{(2,0)}\,=~\parbox{0.4cm}{\begin{tikzpicture}[scale=0.6]
\draw[fill=black] (0,0) circle (2pt);
\draw[thick] (0,0) -- (-0.25,-0.5);
\draw[thick] (0,0) -- (0.25,-0.5);
\end{tikzpicture}}\in~ \hom_{+}^-\big(\bbK,\g^{\otimes 2}\big)^{-2}\quad,
\end{flalign}
which satisfies the identities
\begin{flalign}\label{eqn:42Lieidentities}
\partial\,\parbox{0.4cm}{\begin{tikzpicture}[scale=0.6]
\draw[fill=black] (0,0) circle (2pt);
\draw[thick] (0,0) -- (-0.25,-0.5);
\draw[thick] (0,0) -- (0.25,-0.5);
\end{tikzpicture}}~=~0\quad,\qquad 0~=~
\parbox{0.8cm}{\begin{tikzpicture}[scale=0.6]
\draw[fill=black] (0,0) circle (2pt);
\draw[thick] (0,0) -- (-0.25,-0.5);
\draw[thick] (0,0) -- (0.5,-1);
\draw[fill=black] (-0.25,-0.5) circle (2pt);
\draw[thick] (-0.75,0.5) -- (-0.25,-0.5) -- (-0.25,-1);
\draw[thick] (-0.25,-1) -- (-0.25,-2);
\draw[thick] (0.5,-1) -- (0.5,-2);
\draw[thick,fill=white,dotted] (-0.65,-1) rectangle (0.9,-1.7) node[pos=.5] {$\scriptstyle(1,1)^\even$};
\end{tikzpicture}}~\quad.
\end{flalign}
\end{propo}
\begin{proof}
Specializing the degree counting inequality \eqref{eqn:inequality} to the present case of
$N=2$ and $n=4$, one obtains
\begin{flalign}
-m\,\leq\,4(1-m)+2-l = 6-4m-l \,\leq\,l\quad.
\end{flalign}
Since $m\geq 2$ and $l\geq 0$, the first inequality is only satisfied for $(m,l)=(2,0)$,
i.e.\ \eqref{eqn:42Liestructures} is the only non-trivial component of a 
shifted Poisson structure in this case. The identities \eqref{eqn:42Lieidentities}
follow directly from the $(m,l) = (2,0)$ and $(m,l)=(2,1)$ components 
of \eqref{eqn:semifreeidentities}.
\end{proof}

\begin{rem}\label{rem:graphicaltoformula}
One of the advantages of our graphical calculus, when compared with element-wise algebraic expressions, 
is that it conveniently absorbs Koszul signs and easily deals with linear maps
taking values in tensor products, without using cumbersome Sweedler-like notations.
Let us illustrate these points by rewriting the result of Proposition \ref{prop:4shifted2Lie}
in terms of element-wise expressions. Since $\pi^{(2,0)}\in \hom_{+}^-\big(\bbK,\g^{\otimes 2}\big)^{-2}\cong
\big(\Sym^2\g\big)^{-2}$ takes values in a tensor product, one has to introduce
a notation of the form $\pi^{(2,0)} = \pi^{(2,0)}_{(1)}\otimes \pi^{(2,0)}_{(2)}$ 
(summation implicitly understood) to access the two individual legs.
The identities in \eqref{eqn:42Lieidentities} then read in element-wise notation as
\begin{subequations}
\begin{flalign}
\dd \pi^{(2,0)} \,&=\, \Big(\dd \pi^{(2,0)}_{(1)}\Big)\otimes \pi^{(2,0)}_{(2)} + 
(-1)^{\vert \pi^{(2,0)}_{(1)}\vert}~\pi^{(2,0)}_{(1)}\otimes \Big(\dd \pi^{(2,0)}_{(2)}\Big) \,=\, 0 \quad,\\[3pt]
0\,&=\, \Big[x,\pi^{(2,0)}_{(1)} \Big]\otimes \pi^{(2,0)}_{(2)} + 
(-1)^{\big(\vert x\vert + \big\vert \pi^{(2,0)}_{(1)}\big\vert\big)\,\big\vert\pi^{(2,0)}_{(2)}\big\vert}~
\pi^{(2,0)}_{(2)} \otimes \Big[x,\pi^{(2,0)}_{(1)} \Big]\quad,
\end{flalign}
\end{subequations}
for all homogeneous $x\in \g$. While in the current simple example such 
expressions are still acceptable and readable, they will
become increasingly more complicated for the $3$-shifted and $2$-shifted
Poisson structures in Propositions \ref{prop:3shifted2Lie} and \ref{prop:2shifted2Lie} below.
\end{rem}

\begin{propo}\label{prop:3shifted2Lie}
A $3$-shifted Poisson structure on the Chevalley-Eilenberg algebra $\CE(\g)$ of a Lie $2$-algebra
$\g$ with brackets \eqref{eqn:Lie2bracket} is given by the datum of two maps
\begin{flalign}\label{eqn:32Liestructures}
\pi^{(2,0)}\,=~\parbox{0.4cm}{\begin{tikzpicture}[scale=0.6]
\draw[fill=black] (0,0) circle (2pt);
\draw[thick] (0,0) -- (-0.25,-0.5);
\draw[thick] (0,0) -- (0.25,-0.5);
\end{tikzpicture}}\in~ \hom_{-}^-\big(\bbK,\g^{\otimes 2}\big)^{-1}\quad,
\qquad \pi^{(2,1)}\,=~\parbox{0.4cm}{\begin{tikzpicture}[scale=0.6]
\draw[fill=black] (0,0) circle (2pt);
\draw[thick] (0,0) -- (-0.25,-0.5);
\draw[thick] (0,0) -- (0.25,-0.5);
\draw[thick] (0,0.5) -- (0,0);
\end{tikzpicture}}\in~ \hom_{-}^{-}\big(\g,\g^{\otimes 2}\big)^{-2}
\end{flalign}
of degree $-1$ and, respectively, degree $-2$, which satisfy the identities
\begin{subequations}\label{eqn:32Lieidentities}
\begin{flalign}
\partial\,\parbox{0.4cm}{\begin{tikzpicture}[scale=0.6]
\draw[fill=black] (0,0) circle (2pt);
\draw[thick] (0,0) -- (-0.25,-0.5);
\draw[thick] (0,0) -- (0.25,-0.5);
\end{tikzpicture}}~&=~0\quad,\qquad
-\partial\,\parbox{0.4cm}{\begin{tikzpicture}[scale=0.6]
\draw[fill=black] (0,0) circle (2pt);
\draw[thick] (0,0) -- (-0.25,-0.5);
\draw[thick] (0,0) -- (0.25,-0.5);
\draw[thick] (0,0.5) -- (0,0);
\end{tikzpicture}}~=~
\parbox{0.8cm}{\begin{tikzpicture}[scale=0.6]
\draw[fill=black] (0,0) circle (2pt);
\draw[thick] (0,0) -- (-0.25,-0.5);
\draw[thick] (0,0) -- (0.5,-1);
\draw[fill=black] (-0.25,-0.5) circle (2pt);
\draw[thick] (-0.75,0.5) -- (-0.25,-0.5) -- (-0.25,-1);
\draw[thick] (-0.25,-1) -- (-0.25,-2);
\draw[thick] (0.5,-1) -- (0.5,-2);
\draw[thick,fill=white,dotted] (-0.65,-1) rectangle (0.9,-1.7) node[pos=.5] {$\scriptstyle(1,1)^\odd$};
\end{tikzpicture}}~\quad,\\[5pt]
0~&=~\parbox{0.4cm}{\begin{tikzpicture}[scale=0.7]
\draw[fill=black] (0,0) circle (2pt);
\draw[thick] (-0.25,0.5) -- (0,0);
\draw[thick] (0.25,0.5) -- (0,0);
\draw[thick] (0,0) -- (0,-0.5);
\draw[fill=black] (0,-0.5) circle (2pt);
\draw[thick] (0,-0.5) -- (-0.25,-1);
\draw[thick] (0,-0.5) -- (0.25,-1);
\end{tikzpicture}}~-~
\parbox{0.9cm}{\begin{tikzpicture}[scale=0.6]
\draw[fill=black] (0,0) circle (2pt);
\draw[thick] (0,1.5)-- (0,0.5) -- (0,0);
\draw[thick] (0,0) -- (-0.25,-0.5);
\draw[thick] (0,0) -- (0.5,-1) -- (0.5,-2);
\draw[fill=black] (-0.25,-0.5) circle (2pt);
\draw[thick] (-0.25,-0.5) -- (-0.75,0.5) -- (-0.75,1.5);
\draw[thick] (-0.25,-0.5) -- (-0.25,-2) ;
\draw[thick,fill=white,dotted] (-0.65,-1) rectangle (0.9,-1.7) node[pos=.5] {$\scriptstyle(1,1)^\odd$};
\draw[thick,fill=white,dotted] (-1.15,0.5) rectangle (0.4,1.2) node[pos=.5] {$\scriptstyle(1,1)^\odd$};
\end{tikzpicture}} ~~~~+~
\parbox{0.9cm}{\begin{tikzpicture}[scale=0.6]
\draw[fill=black] (0.15,0) circle (2pt);
\draw[thick] (0.15,0) -- (-0.25,-0.5);
\draw[thick] (0.15,0) -- (0.5,-1) -- (0.5,-2);
\draw[fill=black] (-0.25,-0.5) circle (2pt);
\draw[thick] (-0.25,-0.5) -- (-0.75,0.5);
\draw[thick] (-0.25,-0.5) -- (-0.25,0.5);
\draw[thick] (-0.25,-0.5) -- (-0.25,-2) ;
\draw[thick,fill=white,dotted] (-0.65,-1) rectangle (0.9,-1.7) node[pos=.5] {$\scriptstyle(1,1)^\odd$};
\end{tikzpicture}}~\quad,\\[5pt]
0~&=~\parbox{0.9cm}{\begin{tikzpicture}[scale=0.6]
\draw[fill=black] (0,0) circle (2pt);
\draw[thick] (0,0) -- (-0.25,-0.5);
\draw[thick] (0,0) -- (0.5,-1);
\draw[fill=black] (-0.25,-0.5) circle (2pt);
\draw[thick] (-0.25,-0.5) -- (-0.5,-1);
\draw[thick] (-0.25,-0.5) -- (0,-1);
\draw[thick] (-0.5,-1) -- (-0.5,-2);
\draw[thick] (0,-1) -- (0,-2);
\draw[thick] (0.5,-1) -- (0.5,-2);
\draw[thick,fill=white,dotted] (-0.8,-1) rectangle (0.8,-1.7) node[pos=.5] {$\scriptstyle(2,1)^\odd$};
\end{tikzpicture}}~\quad.
\end{flalign}
\end{subequations}
\end{propo}
\begin{proof}
Specializing the degree counting inequality \eqref{eqn:inequality} to the present case of
$N=2$ and $n=3$, one obtains
\begin{flalign}
-m\,\leq\,3(1-m)+2-l = 5-3m-l \,\leq\,l\quad.
\end{flalign}
Since $m\geq 2$ and $l\geq 0$, the first inequality is only satisfied for $(m,l)=(2,0)$
and for $(m,l)=(2,1)$,
i.e.\ \eqref{eqn:32Liestructures} are the only non-trivial components of a 
shifted Poisson structure in this case. 
Observing that, for degree reasons, $\pi^{(1,3)} : \g^0\otimes\g^0\otimes \g^0 \to \g^{-1}$
and $\pi^{(2,1)}: \g^0\to \g^{-1}\otimes \g^{-1}$, one obtains \eqref{eqn:ternarysquare} and
\begin{flalign}\label{eqn:32degreevanishing}
0~=~
\parbox{0.8cm}{\begin{tikzpicture}[scale=0.6]
\draw[fill=black] (0,0) circle (2pt);
\draw[thick] (0,0) -- (0,0.5);
\draw[thick] (0.5,-1) -- (0,0);
\draw[thick] (-0.25,-0.5) -- (0,0);
\draw[fill=black] (-0.25,-0.5) circle (2pt);
\draw[thick]  (-0.25,-0.5) -- (0,-1);
\draw[thick]  (-0.25,-0.5) -- (-0.5,-1);
\draw[thick] (0.5,-1) -- (0.5,-2);
\draw[thick] (0,-1) -- (0,-2);
\draw[thick] (-0.5,-1) -- (-0.5,-2);
\draw[thick,fill=white,dotted] (-0.8,-1) rectangle (0.8,-1.7) node[pos=.5] {$\scriptstyle(2,1)^\odd$};
\end{tikzpicture}}~~\quad,\qquad
0~=~\parbox{0.4cm}{\begin{tikzpicture}[scale=0.7]
\draw[fill=black] (0,0) circle (2pt);
\draw[thick] (-0.25,0.5) -- (0,0);
\draw[thick] (0,0.5) -- (0,0);
\draw[thick] (0.25,0.5) -- (0,0);
\draw[thick] (0,0) -- (0,-0.5);
\draw[fill=black] (0,-0.5) circle (2pt);
\draw[thick] (0,-0.5) -- (-0.25,-1);
\draw[thick] (0,-0.5) -- (0.25,-1);
\end{tikzpicture}}\quad,\qquad 0~=~
\parbox{0.9cm}{\begin{tikzpicture}[scale=0.6]
\draw[fill=black] (0,0) circle (2pt);
\draw[thick] (0,1.5)-- (0,0.5) -- (0,0);
\draw[thick] (-0.25,1.5)-- (-0.25,-0.5);
\draw[thick] (0,0) -- (-0.25,-0.5);
\draw[thick] (0,0) -- (0.5,-1) -- (0.5,-2);
\draw[fill=black] (-0.25,-0.5) circle (2pt);
\draw[thick] (-0.25,-0.5) -- (-0.75,0.5) -- (-0.75,1.5);
\draw[thick] (-0.25,-0.5) -- (-0.25,-2) ;
\draw[thick,fill=white,dotted] (-0.65,-1) rectangle (0.9,-1.7) node[pos=.5] {$\scriptstyle(1,1)^\odd$};
\draw[thick,fill=white,dotted] (-1.15,0.5) rectangle (0.4,1.2) node[pos=.5] {$\scriptstyle(2,1)^\odd$};
\end{tikzpicture}}~~\quad.
\end{flalign}
The identities \eqref{eqn:32Lieidentities}
then follow directly from the $(m,l) = (2,0)$, $(2,1)$, $(2,2)$ and $(3,0)$ 
components of \eqref{eqn:semifreeidentities}.
\end{proof}

\begin{propo}\label{prop:2shifted2Lie}
A $2$-shifted Poisson structure on the Chevalley-Eilenberg algebra $\CE(\g)$ of a Lie $2$-algebra
$\g$ with brackets \eqref{eqn:Lie2bracket} is given by the datum of one degree $0$ map
\begin{subequations}\label{eqn:22Liestructures}
\begin{flalign}
\pi^{(2,0)}\,=~\parbox{0.4cm}{\begin{tikzpicture}[scale=0.6]
\draw[fill=black] (0,0) circle (2pt);
\draw[thick] (0,0) -- (-0.25,-0.5);
\draw[thick] (0,0) -- (0.25,-0.5);
\end{tikzpicture}}\in~ \hom_{+}^-\big(\bbK,\g^{\otimes 2}\big)^{0}\quad,
\end{flalign}
one degree $-1$ map
\begin{flalign}
\pi^{(2,1)}\,=~\parbox{0.4cm}{\begin{tikzpicture}[scale=0.6]
\draw[fill=black] (0,0) circle (2pt);
\draw[thick] (0,0) -- (-0.25,-0.5);
\draw[thick] (0,0) -- (0.25,-0.5);
\draw[thick] (0,0.5) -- (0,0);
\end{tikzpicture}}\in~ \hom_{+}^{-}\big(\g,\g^{\otimes 2}\big)^{-1}\quad,
\end{flalign}
two degree $-2$ maps
\begin{flalign}
\pi^{(2,2)}\,=~\parbox{0.4cm}{\begin{tikzpicture}[scale=0.6]
\draw[fill=black] (0,0) circle (2pt);
\draw[thick] (0,0) -- (-0.25,-0.5);
\draw[thick] (0,0) -- (0.25,-0.5);
\draw[thick] (-0.25,0.5) -- (0,0);
\draw[thick] (0.25,0.5) -- (0,0);
\end{tikzpicture}}\in~ \hom_{+}^{-}\big(\g^{\otimes 2},\g^{\otimes 2}\big)^{-2}
\quad,\qquad
\pi^{(3,0)}\,=~\parbox{0.55cm}{\begin{tikzpicture}[scale=0.6]
\draw[fill=black] (0,0) circle (2pt);
\draw[thick] (0,0) -- (-0.35,-0.5);
\draw[thick] (0,0) -- (0,-0.5);
\draw[thick] (0,0) -- (0.35,-0.5);
\end{tikzpicture}}\in~ \hom_{+}^-\big(\bbK,\g^{\otimes 3}\big)^{-2}\quad,
\end{flalign}
one degree $-3$ map
\begin{flalign}
\pi^{(3,1)}\,=~\parbox{0.55cm}{\begin{tikzpicture}[scale=0.6]
\draw[fill=black] (0,0) circle (2pt);
\draw[thick] (0,0) -- (-0.35,-0.5);
\draw[thick] (0,0) -- (0,-0.5);
\draw[thick] (0,0) -- (0.35,-0.5);
\draw[thick] (0,0) -- (0,0.5);
\end{tikzpicture}}\in~ \hom_{+}^-\big(\g,\g^{\otimes 3}\big)^{-3}\quad,
\end{flalign}
and one degree $-4$ map
\begin{flalign}
\pi^{(4,0)}\,=~\parbox{0.65cm}{\begin{tikzpicture}[scale=0.6]
\draw[fill=black] (0,0) circle (2pt);
\draw[thick] (0,0) -- (-0.525,-0.5);
\draw[thick] (0,0) -- (-0.175,-0.5);
\draw[thick] (0,0) -- (0.175,-0.5);
\draw[thick] (0,0) -- (0.525,-0.5);
\end{tikzpicture}}\in~ \hom_{+}^-\big(\bbK,\g^{\otimes 4}\big)^{-4}\quad.
\end{flalign}
\end{subequations}
These data have to satisfy the following nine identities:
\begin{subequations}\label{eqn:22Lieidentities}
\begin{flalign}
-\partial\,\parbox{0.4cm}{\begin{tikzpicture}[scale=0.6]
\draw[fill=black] (0,0) circle (2pt);
\draw[thick] (0,0) -- (-0.25,-0.5);
\draw[thick] (0,0) -- (0.25,-0.5);
\draw[thick] (0,0.5) -- (0,0);
\end{tikzpicture}}~&=~
\parbox{0.8cm}{\begin{tikzpicture}[scale=0.6]
\draw[fill=black] (0,0) circle (2pt);
\draw[thick] (0,0) -- (-0.25,-0.5);
\draw[thick] (0,0) -- (0.5,-1);
\draw[fill=black] (-0.25,-0.5) circle (2pt);
\draw[thick] (-0.75,0.5) -- (-0.25,-0.5) -- (-0.25,-1);
\draw[thick] (-0.25,-1) -- (-0.25,-2);
\draw[thick] (0.5,-1) -- (0.5,-2);
\draw[thick,fill=white,dotted] (-0.65,-1) rectangle (0.9,-1.7) node[pos=.5] {$\scriptstyle(1,1)^\even$};
\end{tikzpicture}}~\quad,\\[5pt]
-\partial\,\parbox{0.4cm}{\begin{tikzpicture}[scale=0.6]
\draw[fill=black] (0,0) circle (2pt);
\draw[thick] (0,0) -- (-0.25,-0.5);
\draw[thick] (0,0) -- (0.25,-0.5);
\draw[thick] (-0.25,0.5) -- (0,0);
\draw[thick] (0.25,0.5) -- (0,0);
\end{tikzpicture}}~&=~\parbox{0.4cm}{\begin{tikzpicture}[scale=0.7]
\draw[fill=black] (0,0) circle (2pt);
\draw[thick] (-0.25,0.5) -- (0,0);
\draw[thick] (0.25,0.5) -- (0,0);
\draw[thick] (0,0) -- (0,-0.5);
\draw[fill=black] (0,-0.5) circle (2pt);
\draw[thick] (0,-0.5) -- (-0.25,-1);
\draw[thick] (0,-0.5) -- (0.25,-1);
\end{tikzpicture}}~~-~
\parbox{0.9cm}{\begin{tikzpicture}[scale=0.6]
\draw[fill=black] (0,0) circle (2pt);
\draw[thick] (0,1.5)-- (0,0.5) -- (0,0);
\draw[thick] (0,0) -- (-0.25,-0.5);
\draw[thick] (0,0) -- (0.5,-1) -- (0.5,-2);
\draw[fill=black] (-0.25,-0.5) circle (2pt);
\draw[thick] (-0.25,-0.5) -- (-0.75,0.5) -- (-0.75,1.5);
\draw[thick] (-0.25,-0.5) -- (-0.25,-2) ;
\draw[thick,fill=white,dotted] (-0.65,-1) rectangle (0.9,-1.7) node[pos=.5] {$\scriptstyle(1,1)^\even$};
\draw[thick,fill=white,dotted] (-1.15,0.5) rectangle (0.4,1.2) node[pos=.5] {$\scriptstyle(1,1)^\odd$};
\end{tikzpicture}}~~~~+~
\parbox{0.9cm}{\begin{tikzpicture}[scale=0.6]
\draw[fill=black] (0.15,0) circle (2pt);
\draw[thick] (0.15,0) -- (-0.25,-0.5);
\draw[thick] (0.15,0) -- (0.5,-1) -- (0.5,-2);
\draw[fill=black] (-0.25,-0.5) circle (2pt);
\draw[thick] (-0.25,-0.5) -- (-0.75,0.5);
\draw[thick] (-0.25,-0.5) -- (-0.25,0.5);
\draw[thick] (-0.25,-0.5) -- (-0.25,-2) ;
\draw[thick,fill=white,dotted] (-0.65,-1) rectangle (0.9,-1.7) node[pos=.5] {$\scriptstyle(1,1)^\even$};
\end{tikzpicture}}~\quad,\\[5pt]
0~&=~\parbox{0.4cm}{\begin{tikzpicture}[scale=0.7]
\draw[fill=black] (0,0) circle (2pt);
\draw[thick] (-0.25,0.5) -- (0,0);
\draw[thick] (0,0.5) -- (0,0);
\draw[thick] (0.25,0.5) -- (0,0);
\draw[thick] (0,0) -- (0,-0.5);
\draw[fill=black] (0,-0.5) circle (2pt);
\draw[thick] (0,-0.5) -- (-0.25,-1);
\draw[thick] (0,-0.5) -- (0.25,-1);
\end{tikzpicture}}~~+~
\parbox{0.9cm}{\begin{tikzpicture}[scale=0.6]
\draw[fill=black] (0,0) circle (2pt);
\draw[thick] (0.25,1.5)-- (0.25,0.5) -- (0,0);
\draw[thick] (-0.25,1.5)-- (-0.25,0.5) -- (0,0);
\draw[thick] (0,0) -- (-0.25,-0.5);
\draw[thick] (0,0) -- (0.5,-1) -- (0.5,-2);
\draw[fill=black] (-0.25,-0.5) circle (2pt);
\draw[thick] (-0.25,-0.5) -- (-0.75,0.5) -- (-0.75,1.5);
\draw[thick] (-0.25,-0.5) -- (-0.25,-2) ;
\draw[thick,fill=white,dotted] (-0.65,-1) rectangle (0.9,-1.7) node[pos=.5] {$\scriptstyle(1,1)^\even$};
\draw[thick,fill=white,dotted] (-1.15,0.5) rectangle (0.45,1.2) node[pos=.5] {$\scriptstyle(1,2)^\odd$};
\end{tikzpicture}}
~~~~+~
\parbox{0.8cm}{\begin{tikzpicture}[scale=0.6]
\draw[fill=black] (0,0) circle (2pt);
\draw[thick] (0,0) -- (-0.25,-0.5);
\draw[thick] (0,0) -- (0.25,-0.5);
\draw[thick] (-0.5,1) -- (0,0);
\draw[thick] (0.25,0.5) -- (0,0);
\draw[fill=black] (0.25,0.5) circle (2pt);
\draw[thick]  (0.25,0.5) -- (0,1);
\draw[thick]  (0.25,0.5) -- (0.5,1);
\draw[thick] (-0.5,1) -- (-0.5,2);
\draw[thick] (0,1) -- (0,2);
\draw[thick] (0.5,1) -- (0.5,2);
\draw[thick,fill=white,dotted] (-0.8,1) rectangle (0.8,1.7) node[pos=.5] {$\scriptstyle(1,2)^\odd$};
\end{tikzpicture}}
~~~~+~
\parbox{0.9cm}{\begin{tikzpicture}[scale=0.6]
\draw[fill=black] (0,0) circle (2pt);
\draw[thick] (0,1.5)-- (0,0.5) -- (0,0);
\draw[thick] (-0.25,1.5)-- (-0.25,-0.5);
\draw[thick] (0,0) -- (-0.25,-0.5);
\draw[thick] (0,0) -- (0.5,-1) -- (0.5,-2);
\draw[fill=black] (-0.25,-0.5) circle (2pt);
\draw[thick] (-0.25,-0.5) -- (-0.75,0.5) -- (-0.75,1.5);
\draw[thick] (-0.25,-0.5) -- (-0.25,-2) ;
\draw[thick,fill=white,dotted] (-0.65,-1) rectangle (0.9,-1.7) node[pos=.5] {$\scriptstyle(1,1)^\even$};
\draw[thick,fill=white,dotted] (-1.15,0.5) rectangle (0.4,1.2) node[pos=.5] {$\scriptstyle(2,1)^\odd$};
\end{tikzpicture}}~~\quad,\\[5pt]
-\partial\,\parbox{0.55cm}{\begin{tikzpicture}[scale=0.6]
\draw[fill=black] (0,0) circle (2pt);
\draw[thick] (0,0) -- (-0.35,-0.5);
\draw[thick] (0,0) -- (0,-0.5);
\draw[thick] (0,0) -- (0.35,-0.5);
\end{tikzpicture}} ~&=~\parbox{0.8cm}{\begin{tikzpicture}[scale=0.6]
\draw[fill=black] (0,0) circle (2pt);
\draw[thick] (0,0) -- (-0.25,-0.5);
\draw[thick] (0,0) -- (0.5,-1);
\draw[fill=black] (-0.25,-0.5) circle (2pt);
\draw[thick] (-0.25,-0.5) -- (-0.5,-1);
\draw[thick] (-0.25,-0.5) -- (0,-1);
\draw[thick] (-0.5,-1) -- (-0.5,-2);
\draw[thick] (0,-1) -- (0,-2);
\draw[thick] (0.5,-1) -- (0.5,-2);
\draw[thick,fill=white,dotted] (-0.8,-1) rectangle (0.8,-1.7) node[pos=.5] {$\scriptstyle(2,1)^\even$};
\end{tikzpicture}}~~\quad,\\[5pt]
\label{eqn:22Lieidentities31}
-\partial\,\parbox{0.55cm}{\begin{tikzpicture}[scale=0.6]
\draw[fill=black] (0,0) circle (2pt);
\draw[thick] (0,0) -- (-0.35,-0.5);
\draw[thick] (0,0) -- (0,-0.5);
\draw[thick] (0,0) -- (0.35,-0.5);
\draw[thick] (0,0) -- (0,0.5);
\end{tikzpicture}} ~&=~
\parbox{0.8cm}{\begin{tikzpicture}[scale=0.6]
\draw[fill=black] (0,0) circle (2pt);
\draw[thick] (0,0) -- (0,0.5);
\draw[thick] (0.5,-1) -- (0,0);
\draw[thick] (-0.25,-0.5) -- (0,0);
\draw[fill=black] (-0.25,-0.5) circle (2pt);
\draw[thick]  (-0.25,-0.5) -- (0,-1);
\draw[thick]  (-0.25,-0.5) -- (-0.5,-1);
\draw[thick] (0.5,-1) -- (0.5,-2);
\draw[thick] (0,-1) -- (0,-2);
\draw[thick] (-0.5,-1) -- (-0.5,-2);
\draw[thick,fill=white,dotted] (-0.8,-1) rectangle (0.8,-1.7) node[pos=.5] {$\scriptstyle(2,1)^\even$};
\end{tikzpicture}}~~~+~
\parbox{0.8cm}{\begin{tikzpicture}[scale=0.6]
\draw[fill=black] (0,0) circle (2pt);
\draw[thick] (0,0) -- (-0.25,-0.5);
\draw[thick] (0,0) -- (0.5,-1);
\draw[thick] (0,0) -- (0,-2);
\draw[fill=black] (-0.25,-0.5) circle (2pt);
\draw[thick] (-0.75,0.5) -- (-0.25,-0.5) -- (-0.25,-1);
\draw[thick] (-0.25,-1) -- (-0.25,-2);
\draw[thick] (0.5,-1) -- (0.5,-2);
\draw[thick,fill=white,dotted] (-0.65,-1) rectangle (0.9,-1.7) node[pos=.5] {$\scriptstyle(1,2)^\even$};
\end{tikzpicture}}~~~+~
\parbox{0.8cm}{\begin{tikzpicture}[scale=0.6]
\draw[fill=black] (0,0) circle (2pt);
\draw[thick] (0,0) -- (-0.25,-0.5);
\draw[thick] (0,0) -- (0.5,-1);
\draw[fill=black] (-0.25,-0.5) circle (2pt);
\draw[thick] (-0.75,0.5) -- (-0.25,-0.5) -- (-0.5,-1);
\draw[thick]  (-0.25,-0.5) -- (0,-1);
\draw[thick] (-0.5,-1) -- (-0.5,-2);
\draw[thick] (0,-1) -- (0,-2);
\draw[thick] (0.5,-1) -- (0.5,-2);
\draw[thick,fill=white,dotted] (-0.8,-1) rectangle (0.8,-1.7) node[pos=.5] {$\scriptstyle(2,1)^\even$};
\end{tikzpicture}}~~\quad,\\[5pt]
0~&=~\parbox{0.4cm}{\begin{tikzpicture}[scale=0.7]
\draw[fill=black] (0,0) circle (2pt);
\draw[thick] (-0.25,0.5) -- (0,0);
\draw[thick] (0.25,0.5) -- (0,0);
\draw[thick] (0,0) -- (0,-0.5);
\draw[fill=black] (0,-0.5) circle (2pt);
\draw[thick] (0,-0.5) -- (-0.25,-1);
\draw[thick] (0,-0.5) -- (0,-1);
\draw[thick] (0,-0.5) -- (0.25,-1);
\end{tikzpicture}}~~+~
\parbox{0.8cm}{\begin{tikzpicture}[scale=0.6]
\draw[fill=black] (0,0) circle (2pt);
\draw[thick] (0,0) -- (-0.25,0.5);
\draw[thick] (0,0) -- (0.25,0.5);
\draw[thick] (0.5,-1) -- (0,0);
\draw[thick] (-0.25,-0.5) -- (0,0);
\draw[fill=black] (-0.25,-0.5) circle (2pt);
\draw[thick]  (-0.25,-0.5) -- (0,-1);
\draw[thick]  (-0.25,-0.5) -- (-0.5,-1);
\draw[thick] (0.5,-1) -- (0.5,-2);
\draw[thick] (0,-1) -- (0,-2);
\draw[thick] (-0.5,-1) -- (-0.5,-2);
\draw[thick,fill=white,dotted] (-0.8,-1) rectangle (0.8,-1.7) node[pos=.5] {$\scriptstyle(2,1)^\even$};
\end{tikzpicture}}~~~-~
\parbox{0.9cm}{\begin{tikzpicture}[scale=0.6]
\draw[fill=black] (0,0) circle (2pt);
\draw[thick] (0,1.5)-- (0,0.5) -- (0,0);
\draw[thick] (0,0) -- (-0.25,-0.5);
\draw[thick] (0,0) -- (0.5,-1) -- (0.5,-2);
\draw[thick] (0,0)  -- (0,-2);
\draw[fill=black] (-0.25,-0.5) circle (2pt);
\draw[thick] (-0.25,-0.5) -- (-0.75,0.5) -- (-0.75,1.5);
\draw[thick] (-0.25,-0.5) -- (-0.25,-2) ;
\draw[thick,fill=white,dotted] (-0.65,-1) rectangle (0.9,-1.7) node[pos=.5] {$\scriptstyle(1,2)^\even$};
\draw[thick,fill=white,dotted] (-1.15,0.5) rectangle (0.4,1.2) node[pos=.5] {$\scriptstyle(1,1)^\odd$};
\end{tikzpicture}}~~~~-~
\parbox{0.9cm}{\begin{tikzpicture}[scale=0.6]
\draw[fill=black] (0,0) circle (2pt);
\draw[thick] (0,1.5)-- (0,0.5) -- (0,0);
\draw[thick] (0,0) -- (-0.25,-0.5);
\draw[thick] (0,0) -- (0.5,-1) -- (0.5,-2);
\draw[fill=black] (-0.25,-0.5) circle (2pt);
\draw[thick] (-0.25,-0.5) -- (-0.75,0.5) -- (-0.75,1.5);
\draw[thick] (-0.25,-0.5) -- (-0.5,-1) -- (-0.5,-2) ;
\draw[thick] (-0.25,-0.5) -- (0,-1) -- (0,-2) ;
\draw[thick,fill=white,dotted] (-0.8,-1) rectangle (0.8,-1.7) node[pos=.5] {$\scriptstyle(2,1)^\even$};
\draw[thick,fill=white,dotted] (-1.15,0.5) rectangle (0.4,1.2) node[pos=.5] {$\scriptstyle(1,1)^\odd$};
\end{tikzpicture}}~~~~+~
\parbox{0.8cm}{\begin{tikzpicture}[scale=0.6]
\draw[fill=black] (0,0) circle (2pt);
\draw[thick] (0,0) -- (-0.25,-0.5);
\draw[thick] (-0.25,0.5) -- (-0.25,-0.5);
\draw[thick] (0,0) -- (0.5,-1);
\draw[thick] (0,0) -- (0,-2);
\draw[fill=black] (-0.25,-0.5) circle (2pt);
\draw[thick] (-0.75,0.5) -- (-0.25,-0.5) -- (-0.25,-1);
\draw[thick] (-0.25,-1) -- (-0.25,-2);
\draw[thick] (0.5,-1) -- (0.5,-2);
\draw[thick,fill=white,dotted] (-0.65,-1) rectangle (0.9,-1.7) node[pos=.5] {$\scriptstyle(1,2)^\even$};
\end{tikzpicture}}~~\quad,\\[5pt]
-\partial\,\parbox{0.65cm}{\begin{tikzpicture}[scale=0.6]
\draw[fill=black] (0,0) circle (2pt);
\draw[thick] (0,0) -- (-0.525,-0.5);
\draw[thick] (0,0) -- (-0.175,-0.5);
\draw[thick] (0,0) -- (0.175,-0.5);
\draw[thick] (0,0) -- (0.525,-0.5);
\end{tikzpicture}}~&=~\parbox{0.8cm}{\begin{tikzpicture}[scale=0.6]
\draw[fill=black] (0,0) circle (2pt);
\draw[thick] (0,0) -- (-0.25,-0.5);
\draw[thick] (0,0) -- (0.5,-1);
\draw[fill=black] (-0.25,-0.5) circle (2pt);
\draw[thick] (-0.25,-0.5) -- (-0.5,-1);
\draw[thick] (-0.25,-0.5) -- (0,-1);
\draw[thick] (-0.25,-0.5) -- (-0.25,-2);
\draw[thick] (-0.5,-1) -- (-0.5,-2);
\draw[thick] (0,-1) -- (0,-2);
\draw[thick] (0.5,-1) -- (0.5,-2);
\draw[thick,fill=white,dotted] (-0.8,-1) rectangle (0.8,-1.7) node[pos=.5] {$\scriptstyle(3,1)^\even$};
\end{tikzpicture}}~~~~+~\parbox{0.8cm}{\begin{tikzpicture}[scale=0.6]
\draw[fill=black] (0,0) circle (2pt);
\draw[thick] (0,0) -- (-0.25,-0.5);
\draw[thick] (0,0) -- (0.5,-1);
\draw[thick] (0,0) -- (0,-2);
\draw[fill=black] (-0.25,-0.5) circle (2pt);
\draw[thick] (-0.25,-0.5) -- (-0.5,-1);
\draw[thick] (-0.25,-0.5) -- (-0.25,-1);
\draw[thick] (-0.5,-1) -- (-0.5,-2);
\draw[thick] (-0.25,-1) -- (-0.25,-2);
\draw[thick] (0.5,-1) -- (0.5,-2);
\draw[thick,fill=white,dotted] (-0.8,-1) rectangle (0.8,-1.7) node[pos=.5] {$\scriptstyle(2,2)^\even$};
\end{tikzpicture}}~~\quad,\\[5pt]
0~&=~
\parbox{0.8cm}{\begin{tikzpicture}[scale=0.6]
\draw[fill=black] (0,0) circle (2pt);
\draw[thick] (0,0) -- (0,0.5);
\draw[thick] (0.5,-1) -- (0,0);
\draw[thick] (-0.25,-0.5) -- (0,0);
\draw[fill=black] (-0.25,-0.5) circle (2pt);
\draw[thick]  (-0.25,-0.5) -- (0,-1);
\draw[thick]  (-0.25,-0.5) -- (-0.25,-2);
\draw[thick]  (-0.25,-0.5) -- (-0.5,-1);
\draw[thick] (0.5,-1) -- (0.5,-2);
\draw[thick] (0,-1) -- (0,-2);
\draw[thick] (-0.5,-1) -- (-0.5,-2);
\draw[thick,fill=white,dotted] (-0.8,-1) rectangle (0.8,-1.7) node[pos=.5] {$\scriptstyle(3,1)^\even$};
\end{tikzpicture}}~~~~+~
\parbox{0.8cm}{\begin{tikzpicture}[scale=0.6]
\draw[fill=black] (0,0) circle (2pt);
\draw[thick] (0,0) -- (0,0.5);
\draw[thick] (0,0) -- (0,-2);
\draw[thick] (0.5,-1) -- (0,0);
\draw[thick] (-0.25,-0.5) -- (0,0);
\draw[fill=black] (-0.25,-0.5) circle (2pt);
\draw[thick]  (-0.25,-0.5) -- (-0.25,-1);
\draw[thick]  (-0.25,-0.5) -- (-0.5,-1);
\draw[thick] (0.5,-1) -- (0.5,-2);
\draw[thick] (-0.25,-1) -- (-0.25,-2);
\draw[thick] (-0.5,-1) -- (-0.5,-2);
\draw[thick,fill=white,dotted] (-0.8,-1) rectangle (0.8,-1.7) node[pos=.5] {$\scriptstyle(2,2)^\even$};
\end{tikzpicture}}~~~~+~
\parbox{0.8cm}{\begin{tikzpicture}[scale=0.6]
\draw[fill=black] (0,0) circle (2pt);
\draw[thick] (0,0) -- (-0.25,-0.5);
\draw[thick] (0,0) -- (0.25,-1);
\draw[thick] (0,0) -- (0.5,-1);
\draw[thick] (0,0) -- (0,-2);
\draw[fill=black] (-0.25,-0.5) circle (2pt);
\draw[thick] (-0.75,0.5) -- (-0.25,-0.5) -- (-0.25,-1);
\draw[thick] (-0.25,-1) -- (-0.25,-2);
\draw[thick] (0.25,-1) -- (0.25,-2);
\draw[thick] (0.5,-1) -- (0.5,-2);
\draw[thick,fill=white,dotted] (-0.65,-1) rectangle (0.9,-1.7) node[pos=.5] {$\scriptstyle(1,3)^\even$};
\end{tikzpicture}}~~~~+~
\parbox{0.8cm}{\begin{tikzpicture}[scale=0.6]
\draw[fill=black] (0,0) circle (2pt);
\draw[thick] (0,0) -- (-0.25,-0.5);
\draw[thick] (0,0) -- (0.25,-1) -- (0.25,-2);
\draw[thick] (0,0) -- (0.5,-1);
\draw[fill=black] (-0.25,-0.5) circle (2pt);
\draw[thick] (-0.75,0.5) -- (-0.25,-0.5) -- (-0.5,-1);
\draw[thick]  (-0.25,-0.5) -- (-0.25,-1);
\draw[thick] (-0.5,-1) -- (-0.5,-2);
\draw[thick] (-0.25,-1) -- (-0.25,-2);
\draw[thick] (0.5,-1) -- (0.5,-2);
\draw[thick,fill=white,dotted] (-0.8,-1) rectangle (0.8,-1.7) node[pos=.5] {$\scriptstyle(2,2)^\even$};
\end{tikzpicture}}~~\quad,\\[5pt]
0~&=~\parbox{0.8cm}{\begin{tikzpicture}[scale=0.6]
\draw[fill=black] (0,0) circle (2pt);
\draw[thick] (0,0) -- (-0.25,-0.5);
\draw[thick] (0,0) -- (0.25,-1) -- (0.25,-2);
\draw[thick] (0,0) -- (0.5,-1);
\draw[fill=black] (-0.25,-0.5) circle (2pt);
\draw[thick] (-0.25,-0.5) -- (-0.5,-1);
\draw[thick] (-0.25,-0.5) -- (0,-1);
\draw[thick] (-0.25,-0.5) -- (-0.25,-2);
\draw[thick] (-0.5,-1) -- (-0.5,-2);
\draw[thick] (0,-1) -- (0,-2);
\draw[thick] (0.5,-1) -- (0.5,-2);
\draw[thick,fill=white,dotted] (-0.8,-1) rectangle (0.8,-1.7) node[pos=.5] {$\scriptstyle(3,2)^\even$};
\end{tikzpicture}}~~~~+~\parbox{0.8cm}{\begin{tikzpicture}[scale=0.6]
\draw[fill=black] (0,0) circle (2pt);
\draw[thick] (0,0) -- (-0.25,-0.5);
\draw[thick] (0,0) -- (0.5,-1);
\draw[thick] (0,0) -- (0,-2);
\draw[thick] (0,0) -- (0.25,-1) -- (0.25,-2);
\draw[fill=black] (-0.25,-0.5) circle (2pt);
\draw[thick] (-0.25,-0.5) -- (-0.5,-1);
\draw[thick] (-0.25,-0.5) -- (-0.25,-1);
\draw[thick] (-0.5,-1) -- (-0.5,-2);
\draw[thick] (-0.25,-1) -- (-0.25,-2);
\draw[thick] (0.5,-1) -- (0.5,-2);
\draw[thick,fill=white,dotted] (-0.8,-1) rectangle (0.8,-1.7) node[pos=.5] {$\scriptstyle(2,3)^\even$};
\end{tikzpicture}}~~\quad.
\end{flalign}
\end{subequations}
\end{propo}
\begin{proof}
Specializing the degree counting inequality \eqref{eqn:inequality} to the present case of
$N=2$ and $n=2$, one obtains
\begin{flalign}
-m\,\leq\,2(1-m)+2-l = 4-2m-l \,\leq\,l\quad.
\end{flalign}
Since $m\geq 2$ and $l\geq 0$, the first inequality is only satisfied for $(m,l)=(2,0)$,
$(2,1)$, $(2,2)$, $(3,0)$, $(3,1)$ and $(4,0)$,
i.e.\ \eqref{eqn:22Liestructures} are the only non-trivial components of a 
shifted Poisson structure in this case. In order to reduce the family of identities
\eqref{eqn:semifreeidentities} to \eqref{eqn:22Lieidentities}, one uses that
various terms vanish as a consequence of the specific form of the maps
$\pi^{(1,3)} : \g^0\otimes\g^0\otimes\g^0\to\g^{-1}$, 
$\pi^{(2,0)} : \bbK\to \g^0\otimes\g^0$, 
$\pi^{(2,2)} : \g^0\otimes\g^0\to \g^{-1}\otimes\g^{-1}$, 
$\pi^{(3,1)} : \g^{0}\to \g^{-1}\otimes\g^{-1}\otimes\g^{-1}$
and $\pi^{(4,0)}:\bbK\to \g^{-1}\otimes\g^{-1}\otimes\g^{-1}\otimes\g^{-1}$.
\end{proof}

\begin{rem}\label{rem:1shifted2Lie}
In contrast to the finite characterizations of $n=2,3,4$-shifted Poisson 
structures given above, a $1$-shifted Poisson structure on the 
Chevalley-Eilenberg algebra $\CE(\g)$ of a Lie $2$-algebra $\g$ 
with brackets \eqref{eqn:Lie2bracket} involves in general an 
infinite number of data. Indeed, specializing the degree 
counting inequality \eqref{eqn:inequality} to the present case of
$N=2$ and $n=1$, one obtains
\begin{flalign}
-m\,\leq\,(1-m)+2-l = 3-m-l \,\leq\,l\quad.
\end{flalign}
Analyzing these inequalities, one finds that a 
$1$-shifted Poisson structure consists of 
(a priori non-vanishing) maps
\begin{subequations}
\begin{flalign}
\pi^{(2,l)}\,=\!\!\parbox{1.4cm}{\begin{tikzpicture}[scale=0.6]
\draw[thick] (-0.5,0.75) -- (0,0);
\draw[thick] (0,0.75) -- (0,0);
\draw[thick] (0.5,0.75) -- (0,0);
\draw[thick] (0,0) -- (-0.25,-0.75);
\draw[thick] (0,0) -- (0.25,-0.75);
\draw[fill=black] (0,0) circle (2pt);
\draw[thick, decorate,decoration={brace,amplitude=5pt}]
  (-0.5,0.85) -- (0.5,0.85) node[midway,yshift=1em]{\text{\footnotesize{$l$ inputs}}};
  \draw[thick, decorate,decoration={brace,amplitude=5pt,mirror}]
  (-0.475,-0.85) -- (0.475,-0.85) node[midway,yshift=-1em]{\text{\footnotesize{$2$ outputs}}};
\end{tikzpicture}}\in~ \hom_{-}^{-}\big(\g^{\otimes l},\g^{\otimes 2}\big)^{1-l}\quad,
\quad \text{for all $l\in\{1,2,3\}$}\quad,
\end{flalign}
and 
\begin{flalign}
\pi^{(m,l)}\,=\!\!\parbox{1.4cm}{\begin{tikzpicture}[scale=0.6]
\draw[thick] (-0.5,0.75) -- (0,0);
\draw[thick] (0,0.75) -- (0,0);
\draw[thick] (0.5,0.75) -- (0,0);
\draw[thick] (0,0) -- (-0.375,-0.75);
\draw[thick] (0,0) -- (-0.125,-0.75);
\draw[thick] (0,0) -- (0.125,-0.75);
\draw[thick] (0,0) -- (0.375,-0.75);
\draw[fill=black] (0,0) circle (2pt);
\draw[thick, decorate,decoration={brace,amplitude=5pt}]
  (-0.5,0.85) -- (0.5,0.85) node[midway,yshift=1em]{\text{\footnotesize{$l$ inputs}}};
\draw[thick, decorate,decoration={brace,amplitude=5pt,mirror}]
  (-0.475,-0.85) -- (0.475,-0.85) node[midway,yshift=-1em]{\text{\footnotesize{$m$ outputs}}};
\end{tikzpicture}}\in~ \hom_{-}^{-}\big(\g^{\otimes l},\g^{\otimes m}\big)^{3-m-l}\quad,
\quad\text{for all $m\geq 3$ and $l\in\{0,1,2,3\}$}\quad.
\end{flalign}
\end{subequations}
These data must satisfy the infinite tower of identities which is obtained by specializing
\eqref{eqn:semifreeidentities} to the present case.
\end{rem}

\subsection{\label{subsec:explicit}Explicit examples}
In this subsection we apply the general characterization 
results for shifted Poisson structures from Subsection \ref{subsec:2Lie}
to specific classes of examples of Lie $2$-algebras. 
We shall focus mostly on the cases of $(n=3)$- and $(n=4)$-shifted
Poisson structures, because these are new phenomena of our Lie $2$-algebraic context
which are not present for ordinary Lie algebras, see Lemma \ref{lem:degreecounting}.
\begin{ex}[Abelian Lie $2$-algebras]\label{ex:Abelian}
Let us recall that, for the ordinary Abelian Lie algebra $\bbK$, there exists a $1$-parameter family
of $2$-shifted Poisson structures on $\CE(\bbK)$
which is given by $\pi^{(2,0)}\in \big(\Sym^2\,\bbK\big)^{0}\cong \bbK$,
see also Proposition \ref{prop:2shiftedLie}. 
In stark contrast to this, the
Abelian Lie $2$-algebra $\bbK[1] = (\bbK\stackrel{0}{\longrightarrow}0)$,
which arises for example in the theory of Abelian gerbes \cite{Brylinski}, does not
admit any non-trivial $n$-shifted Poisson structures for all $n\geq 1$. This claim
can be verified from our characterization results in
Propositions \ref{prop:4shifted2Lie}, \ref{prop:3shifted2Lie}, \ref{prop:2shifted2Lie}
and Remark \ref{rem:1shifted2Lie}, together with the observation that, for every $m\geq 2$,
we have $\hom^-_+(\bbK[1]^{\otimes l}, \bbK[1]^{\otimes m}) \cong 0$
as a consequence of $\Sym^m\big(\bbK[1]\big) \cong \big(\bigwedge^m \bbK\big)[m]\cong 0$
for $m\geq 2$.
\sk

A Lie $2$-algebraic analogue of the $1$-parameter family 
$\pi^{(2,0)}\in\big(\Sym^2\,\bbK\big)^{0}\cong \bbK$
of $2$-shifted Poisson structures on $\CE(\bbK)$ can be found
by considering the $2$-dimensional Abelian Lie $2$-algebra $\bbK^2[1] = (\bbK^2\stackrel{0}{\longrightarrow}0)$.
Then, by Proposition \ref{prop:4shifted2Lie}, there exists a $1$-parameter
family of $4$-shifted Poisson structures on $\CE(\bbK^2[1])$
which is given by $\pi^{(2,0)}\in\Sym^2\big(\bbK^2[1]\big)^{-2}
\cong \big(\bigwedge^2\bbK^2\big)^{0}\cong\bbK$. The required doubling of dimensions
is thus a feature which arises from the odd parity of the non-trivial elements of $\bbK^2[1]$,
which turns symmetric powers into antisymmetric ones. Let us further emphasize
that the shift of such shifted Poisson structures
gets doubled from $n=2$ to $n=4$ when passing from ordinary Lie algebras to Lie $2$-algebras.
\end{ex}

\begin{ex}[String Lie $2$-algebras]\label{ex:String}
Let $\h$ be a Lie algebra with Lie bracket $[\,\cdot\,,\,\cdot\,]:\h\otimes\h\to\h$
and choose any $3$-cocycle $\kappa : \h\otimes\h\otimes\h\to\bbK$. Consider the $2$-term
cochain complex
\begin{subequations}
\begin{flalign}
\h_{\kappa}\,:=\,\bbK[1]\oplus\h\,=\,\Big(
\xymatrix{
\bbK \ar[r]^-{0}~&~\h
}\Big)\,\in\,\Ch
\end{flalign}
and endow it with the Lie $2$-algebra
structure (see \eqref{eqn:Lie2bracket} and \eqref{eqn:Lie2bracketidentities})
which is defined by
\begin{align}
\pi^{(1,2)}\,:\, \h_{\kappa}\otimes\h_{\kappa}~&\longrightarrow~\h_{\kappa}\quad, 
& \pi^{(1,3)}\,:\, \h_{\kappa}\otimes\h_{\kappa}\otimes\h_{\kappa}~&\longrightarrow~
\h_{\kappa}\quad,\\
\nn (x,y)~&\longmapsto~[x,y]\quad, & (x,y,z)~&\longmapsto~\kappa(x,y,z)\quad,\\
\nn (a ,y)~&\longmapsto~0\quad, & (a,y,z)~&\longmapsto~0\quad,\\
\nn (a ,b)~&\longmapsto~0\quad, & (a,b,z)~&\longmapsto~0\quad,\\
\nn & & (a,b,c)~&\longmapsto~0\quad,
\end{align}
\end{subequations}
where $x,y,z\in\h$ and $a,b,c\in \bbK[1]$. Such
Lie $2$-algebras appeared first in \cite[Example 6.10]{BaezCrans}
and they are nowadays called string Lie $2$-algebras. Applying our 
results from Propositions \ref{prop:4shifted2Lie} and 
\ref{prop:3shifted2Lie}, we obtain the following
characterizations of the $(n=3)$- and 
$(n=4)$-shifted Poisson structures on $\CE(\h_{\kappa})$:
\begin{itemize}
\item Every $4$-shifted Poisson structure on $\CE(\h_{\kappa})$
is trivial by the same argument $\big(\Sym^2\,\h_{\kappa}\big)^{-2}\cong
\big(\bigwedge^2\bbK\big)^0\cong 0$ as in Example \ref{ex:Abelian}.

\item A $3$-shifted Poisson structure on $\CE(\h_{\kappa})$
is given by two elements
\begin{flalign}
\pi^{(2,0)}\,\in\, \hom_{-}^{-}\big(\bbK,\h_{\kappa}^{\otimes 2}\big)^{-1}\,\cong\,\h\quad,\qquad
\pi^{(2,1)}\,\in\, \hom_{-}^{-}\big(\h_{\kappa},\h_{\kappa}^{\otimes 2}\big)^{-2}\,\cong\,\h^\ast\quad,
\end{flalign}
where $\h^\ast:= \hom(\h,\bbK)$ denotes the dual of $\h$.
Denoting these elements by
\begin{subequations}\label{eqn:3shiftedstring}
\begin{flalign}\label{eqn:3shiftedstring1}
\oone\,\in\, \h\quad,\qquad \langle\,\cdot\,\rangle \,:\, \h\,\longrightarrow\,\bbK\quad,
\end{flalign}
the identities \eqref{eqn:32Lieidentities} for $3$-shifted Poisson structures reduce
to the properties
\begin{flalign}\label{eqn:3shiftedstring2}
[x,\oone]\,=\,0\quad,\qquad
\langle \oone\rangle\,=\,0\quad,\qquad
\big\langle [x,y]\big\rangle\,=\,-\kappa(x,y,\oone)\quad,
\end{flalign}
\end{subequations}
for all $x,y\in\h$. The first property demands that the distinguished
element $\oone\in \h$ is central in the Lie algebra $\h$ 
and the second property demands that it is annihilated by the linear form
$\langle\,\cdot\,\rangle:\h\to\bbK $. The third property
demands that the linear form $\langle\,\cdot\,\rangle$
annihilates the Lie bracket $[\,\cdot\,,\,\cdot\,]$ on $\h$, up to a violation which is determined 
by the restricted the $3$-cocycle $\kappa(\,\cdot\,,\,\cdot\,,\oone)$ on the distinguished element.
\end{itemize}
For completeness, let us also mention that the $2$-shifted Poisson structures 
on $\CE(\h_{\kappa})$ have already been characterized in \cite[Section 4.2]{KempLaugwitzSchenkel}
and we refer the reader to this reference for details.
Furthermore, by Remark \ref{rem:1shifted2Lie}, the $1$-shifted Poisson structures 
on $\CE(\h_{\kappa})$ consist of infinitely many data
and we did not recognize any notable simplifications in the case of string Lie $2$-algebras.
\sk

Let us conclude this example with some further comments about 
the case of $3$-shifted Poisson structures \eqref{eqn:3shiftedstring}. Assuming that
the underlying Lie algebra $\h$ is semisimple, it follows
that the central element $\oone\in\h$
must necessarily be zero $\oone =0$. The properties in \eqref{eqn:3shiftedstring2}
then simplify to the requirement that $\big\langle [x,y]\big\rangle=0$,
for all $x,y\in\h$. Since $[\h,\h]=\h$ for semisimple Lie algebras, it 
follows that the linear form $\langle\,\cdot\,\rangle=0$ must necessarily vanish.
Hence, any $3$-shifted Poisson structure on $\CE(\h_{\kappa})$ is trivial
in the case where the underlying Lie algebra $\h$ is semisimple.
\sk

For a non-semisimple Lie algebra $\h$, there can exist non-trivial
$3$-shifted Poisson structures. Let us consider the following two extremal cases:
1.)~Suppose that $\oone=0$ is trivial. Then the properties in \eqref{eqn:3shiftedstring2} 
reduce to the $1$-cocycle condition $\langle[\,\cdot\,,\,\cdot\,]\rangle = 0$. Hence, every 
non-trivial $1$-cocycle on $\h$ defines a $3$-shifted Poisson structure with $\oone=0$. As a concrete example,
consider a matrix Lie algebra $\h\subseteq \mathfrak{gl}_N(\bbK)$ and observe
that choosing $\oone=0$ and the matrix trace $\langle\,\cdot\,\rangle = \mathrm{Tr}$
defines a $3$-shifted Poisson structure.
2.)~Suppose that $\langle\,\cdot\,\rangle=0$ is trivial. 
For a non-trivial central element $0\neq \oone\in\h$ to exist,
the Lie algebra $\h$ must be a central extension $0\to \bbK\to \h \to \mathfrak{k}\to 0$. 
The third property in \eqref{eqn:3shiftedstring2} then
requires that the $3$-cocycle $\kappa$ entering the string Lie $2$-algebra $\h_{\kappa}$
is induced by the map $\h \to \mathfrak{k}$ from a $3$-cocycle on $\mathfrak{k}$.
If $\kappa$ is chosen of this form, then there exists a  
$3$-shifted Poisson structure defined by 
the central element $\oone\in\h$ picked out by the map $\bbK\to \h$
and the trivial linear form $\langle\,\cdot\,\rangle=0$.
\end{ex}

\begin{ex}[Shifted cotangent Lie $2$-algebras]\label{ex:cotangent}
Let $\h$ be a Lie algebra with Lie bracket $[\,\cdot\,,\,\cdot\,]:\h\otimes\h\to\h$.
Let us endow its dual $\h^\ast:= \hom(\h,\bbK)$ with the coadjoint action
$\ad^\ast : \h\otimes \h^\ast \to \h^\ast$. The shifted cotangent Lie $2$-algebra
is defined by the $2$-term cochain complex
\begin{subequations}
\begin{flalign}
T^\ast[1]\h\,:=\,\h^\ast[1]\oplus \h \,=\, \Big(
\xymatrix{
\h^\ast \ar[r]^-{0}~&~\h
}\Big)\,\in\,\Ch
\end{flalign}
together with the Lie $2$-algebra
structure (see \eqref{eqn:Lie2bracket} and \eqref{eqn:Lie2bracketidentities})
which is defined by the trivial higher bracket $\pi^{(1,3)}=0$ and
\begin{flalign}
\pi^{(1,2)}\,:\,T^\ast[1]\h\otimes T^\ast[1]\h~&\longrightarrow~T^\ast[1]\h\quad,\\
\nn (x,y)~&\longmapsto~[x,y]\quad,\\
\nn (x,\theta)~&\longmapsto~\ad^\ast_{x}(\theta)\quad,\\
\nn (\omega,\theta)~&\longmapsto~0\quad,
\end{flalign}
\end{subequations}
for all $x,y\in\h$ and $\theta,\omega\in\h^\ast[1]$.
Such Lie $2$-algebras are used for example in the 
context of higher-dimensional Chern-Simons theories which are
based on $2$-connections, see e.g.\ \cite{Zucchini} and \cite{SchenkelVicedo}.
Applying our results from Propositions \ref{prop:4shifted2Lie} and
\ref{prop:3shifted2Lie}, we obtain the following
characterizations of the $(n=3)$- and 
$(n=4)$-shifted Poisson structures on $\CE\big(T^\ast[1]\h\big)$:
\begin{itemize}
\item A $4$-shifted Poisson structure on $\CE\big(T^\ast[1]\h\big)$ is given by an element
\begin{flalign}
\pi^{(2,0)}\,\in\,\hom^{-}_{+}\big(\bbK,(T^\ast[1]\h)^{\otimes 2}\big)^{-2}\,\cong\,\text{${\Mwedge}^2$}\h^\ast
\end{flalign}
which is invariant under the tensor product coadjoint action.

\item A $3$-shifted Poisson structure on $\CE\big(T^\ast[1]\h\big)$ is given by two elements
\begin{subequations}\label{eqn:3shiftedcotangent}
\begin{flalign}
\pi^{(2,0)}\,&\in\,\hom^{-}_{-}\big(\bbK,(T^\ast[1]\h)^{\otimes 2}\big)^{-1}\,\cong\,\h^\ast\otimes\h\quad,\\
\pi^{(2,1)}\,&\in\,\hom_{-}^{-}\big(T^\ast[1]\h,(T^\ast[1]\h)^{\otimes 2}\big)^{-2}\,\cong\,
\hom\big(\h, \Sym^2\,\h^\ast\big)\quad.
\end{flalign}
\end{subequations}
Using that the differential $\partial=0$ is trivial in this example, one checks that the 
identities in the first line of \eqref{eqn:32Lieidentities} are equivalent to the 
invariance condition
\begin{subequations}\label{eqn:3shiftedcotangentconditions}
\begin{flalign}\label{eqn:3shiftedcotangentconditions1}
\big(\ad^\ast_x\otimes\id + \id\otimes\ad_x\big)\pi^{(2,0)}\,=\,0\quad,
\end{flalign}
for all $x\in\h$, where $\ad:=[\,\cdot\,,\,\cdot\,]:\h\otimes\h\to \h$ denotes the adjoint action.
The identity in the second line of \eqref{eqn:32Lieidentities}
is equivalent to the condition
\begin{flalign}\label{eqn:3shiftedcotangentconditions2}
\pi^{(2,1)}\big([x,y]\big)\,=\,\big(\ad^\ast_x \otimes \id + \id\otimes\ad^\ast_x\big)\pi^{(2,1)}(y)
-\big(\ad^\ast_y \otimes \id + \id\otimes\ad^\ast_y\big)\pi^{(2,1)}(x)\quad,
\end{flalign}
for all $x,y\in\h$. Note that this is similar to the properties of a Lie bialgebra structure,
with the difference that $\pi^{(2,1)} : \h\to \Sym^2\,\h^\ast$ takes values in the second
symmetric power of the dual $\h^\ast$, in contrast to the second exterior power of $\h$ 
in the case of Lie bialgebras. The identity in the third line of \eqref{eqn:32Lieidentities}
yields the condition
\begin{flalign}\label{eqn:3shiftedcotangentconditions3}
\sum_{\rho\in\mathrm{Sh}(2,1)} \gamma^{\Vec}_{\rho}\Big(\big(\pi^{(2,1)}\otimes \id\big)\pi^{(2,0)}\Big)\,=\,0\quad,
\end{flalign}
\end{subequations}
where $\gamma^{\Vec}$ denotes the symmetric braiding on the category of vector spaces,
i.e.\ it does \textit{not} involve Koszul signs.
(Note that the odd signs $(-1)^{\vert \rho\vert}$ in the sum over shuffles in \eqref{eqn:32Lieidentities}
get compensated by the Koszul signs in the symmetric 
braiding \eqref{eqn:Chbraiding} for $\h^\ast[1]\otimes \h^\ast[1]\otimes \h^\ast[1]$.)
\end{itemize}

We conclude this example by noting that in the context of higher-dimensional
Chern-Simons theory \cite{Zucchini,SchenkelVicedo} the relevant $3$-shifted Poisson
structures are given by the coevaluation map $\pi^{(2,0)} =\mathrm{coev}(1)\in \h^\ast\otimes\h$
and a trivial $\pi^{(2,1)}=0$. Such $3$-shifted Poisson structures are non-degenerate,
hence they have a corresponding $3$-shifted symplectic structure which enters the 
construction of the action functional. Our characterization 
in \eqref{eqn:3shiftedcotangent} and \eqref{eqn:3shiftedcotangentconditions}
shows that this is only a subclass of the $3$-shifted Poisson structures 
on $\CE\big(T^\ast[1]\h\big)$ and in particular there exists more flexibility 
which is given by the datum of the linear map $\pi^{(2,1)}:\h\to\Sym^2\,\h^\ast$.
\end{ex}

\begin{rem}\label{rem:adjustment}
We observe that there are some vague similarities between 
the datum $\pi^{(2,1)}$ contained in a $3$-shifted Poisson structure 
and the \textit{adjustments} from the theory of 
higher connections, see e.g.\ \cite{Gagliardo}, however there
is no precise agreement between these two concepts.
Referring back to Proposition \ref{prop:3shifted2Lie}, let us consider
the case of a general Lie $2$-algebra $\g$. The datum of a $3$-shifted 
Poisson structure can be identified with linear maps $\pi^{(2,0)}: (\g^{-1})^\ast\to \g^0$ 
and $\pi^{(2,1)} : \g^0\to \g^{-1}\otimes \g^{-1}$. We denote the corresponding transposed maps
by $\pi^{(2,0)\ast}: (\g^{0})^\ast\to \g^{-1}$ and $\pi^{(2,1)\ast}: (\g^{-1})^\ast\otimes (\g^{-1})^\ast\to (\g^{0})^\ast$. 
Under the non-degeneracy condition that $\pi^{(2,0)}$ 
is an isomorphism, we can then form the composite map
\begin{flalign}
\kappa\,:=\,\pi^{(2,0)\ast}\circ \pi^{(2,1)\ast} \circ 
\left(\big(\pi^{(2,0)}\big)^{-1}\otimes \big(\pi^{(2,0)}\big)^{-1}\right)\,:\,
\g^0\otimes\g^0~\longrightarrow~\g^{-1}\quad,
\end{flalign}
which takes the same form as an adjustment datum in the sense of \cite[Eqn.\ (3.7)]{Gagliardo}
on the Lie $2$-algebra $\g$. However, this map $\kappa$ does not in general satisfy
the adjustment conditions from \cite[Eqn.\ (3.9)]{Gagliardo}. An easy way to see this is to note that
there exist $3$-shifted Poisson structures with trivial $\pi^{(2,1)}=0$, see e.g.\ Example \ref{ex:cotangent},
leading to $\kappa = 0$, which is incompatible with the adjustment conditions for any Lie $2$-algebra
$\g$ with non-trivial $\ell_2$ or $\ell_3$ bracket.
\end{rem}


\section*{Acknowledgments}
We would like to thank the anonymous reviewer for useful comments that helped us
to improve our paper.
C.K.\ is supported by an EPSRC doctoral studentship. 
A.S.\ was supported by the Royal Society (UK) through a Royal Society University 
Research Fellowship (URF\textbackslash R\textbackslash 211015)
and Enhancement Grants (RF\textbackslash ERE\textbackslash 210053 and 
RF\textbackslash ERE\textbackslash 231077).


\section*{Data availability statement}
All data generated or analyzed during this study are contained in this document.

\section*{Conflict of interest statement}
On behalf of all authors, the corresponding author states that there is no conflict of interest.



\end{document}